\documentclass[12pt,twoside]{preprint}
\usepackage[margin=1.3in]{geometry}
\usepackage{amssymb}
\usepackage{amsmath}
\usepackage{hyperref}
\usepackage{breakurl}
\usepackage{mhenvs}
\usepackage{mhequ}
\usepackage{mhsymb}
\usepackage{booktabs}
\usepackage{tikz}
\usepackage{mathrsfs}
\usepackage{times}
\usepackage{microtype}
\usepackage{comment}
\usepackage{wasysym}
\usepackage{centernot}

\makeatletter
\pgfdeclareshape{crosscircle}
{
  \inheritsavedanchors[from=circle] % this is nearly a circle
  \inheritanchorborder[from=circle]
  \inheritanchor[from=circle]{north}
  \inheritanchor[from=circle]{north west}
  \inheritanchor[from=circle]{north east}
  \inheritanchor[from=circle]{center}
  \inheritanchor[from=circle]{west}
  \inheritanchor[from=circle]{east}
  \inheritanchor[from=circle]{mid}
  \inheritanchor[from=circle]{mid west}
  \inheritanchor[from=circle]{mid east}
  \inheritanchor[from=circle]{base}
  \inheritanchor[from=circle]{base west}
  \inheritanchor[from=circle]{base east}
  \inheritanchor[from=circle]{south}
  \inheritanchor[from=circle]{south west}
  \inheritanchor[from=circle]{south east}
  \inheritbackgroundpath[from=circle]
  \foregroundpath{
    \centerpoint%
    \pgf@xc=\pgf@x%
    \pgf@yc=\pgf@y%
    \pgfutil@tempdima=\radius%
    \pgfmathsetlength{\pgf@xb}{\pgfkeysvalueof{/pgf/outer xsep}}%  
    \pgfmathsetlength{\pgf@yb}{\pgfkeysvalueof{/pgf/outer ysep}}%  
    \ifdim\pgf@xb<\pgf@yb%
      \advance\pgfutil@tempdima by-\pgf@yb%
    \else%
      \advance\pgfutil@tempdima by-\pgf@xb%
    \fi%
    \pgfpathmoveto{\pgfpointadd{\pgfqpoint{\pgf@xc}{\pgf@yc}}{\pgfqpoint{-0.707107\pgfutil@tempdima}{-0.707107\pgfutil@tempdima}}}
    \pgfpathlineto{\pgfpointadd{\pgfqpoint{\pgf@xc}{\pgf@yc}}{\pgfqpoint{0.707107\pgfutil@tempdima}{0.707107\pgfutil@tempdima}}}
    \pgfpathmoveto{\pgfpointadd{\pgfqpoint{\pgf@xc}{\pgf@yc}}{\pgfqpoint{-0.707107\pgfutil@tempdima}{0.707107\pgfutil@tempdima}}}
    \pgfpathlineto{\pgfpointadd{\pgfqpoint{\pgf@xc}{\pgf@yc}}{\pgfqpoint{0.707107\pgfutil@tempdima}{-0.707107\pgfutil@tempdima}}}
  }
}
\makeatother

\colorlet{symbols}{blue!90!black}
\colorlet{testcolor}{green!60!black}
\colorlet{darkblue}{blue!60!black}
\colorlet{darkgreen}{green!60!black}

\def\symbol#1{\textcolor{symbols}{#1}}
\def\1{\mathbf{\symbol{1}}}
\def\X{\symbol{X}}

\usetikzlibrary{calc}
\usetikzlibrary{shapes.misc}
\usetikzlibrary{shapes.symbols}
\usetikzlibrary{shapes.geometric}
\usetikzlibrary{snakes}
\usetikzlibrary{decorations}
\usetikzlibrary{decorations.markings}

\def\drawx{\draw[-,solid] (-3pt,-3pt) -- (3pt,3pt);\draw[-,solid] (-3pt,3pt) -- (3pt,-3pt);}
\tikzset{
	root/.style={circle,fill=testcolor,inner sep=0pt, minimum size=2mm},
	dot/.style={circle,fill=black,inner sep=0pt,minimum size=1mm},
	int/.style={circle,fill=black,draw=black,inner sep=0pt,minimum size=0.7mm},
	circ/.style={circle,draw=black,inner sep=0pt, minimum size=1mm},
	var/.style={circle,fill=black!10,draw=black,inner sep=0pt, minimum size=2mm},
	dotred/.style={circle,fill=black!50,inner sep=0pt, minimum size=2mm},
	generic/.style={semithick,shorten >=1pt,shorten <=1pt},
	oddfunc/.style={generic, dotted},
	dist/.style={ultra thick,draw=testcolor,shorten >=1pt,shorten <=1pt},
	testfcn/.style={ultra thick,testcolor,shorten >=1pt,shorten <=1pt,<-},
	testfunction/.style={ultra thick,testcolor,shorten >=1pt,shorten <=1pt},
	testfcnx/.style={ultra thick,testcolor,shorten >=1pt,shorten <=1pt,<-,
		postaction={decorate,decoration={markings,mark=at position 0.6 with {\drawx}}}},
	kprime/.style={semithick,shorten >=1pt,shorten <=1pt,densely dashed,->},
	kprimex/.style={semithick,shorten >=1pt,shorten <=1pt,densely dashed,->,
		postaction={decorate,decoration={markings,mark=at position 0.4 with {\drawx}}}},
	kernel/.style={semithick,shorten >=1pt,shorten <=1pt,->,draw=black},
	Pkernel/.style={ultra thick,shorten >=1pt,shorten <=1pt,->,draw=blue},
	PkernelBig/.style={very thick,shorten >=1pt,shorten <=1pt,decorate, draw=blue, decoration={zigzag,amplitude=1.5pt,segment length = 3pt,pre length=2pt,post length=2pt}},
	multx/.style={shorten >=1pt,shorten <=1pt,
		postaction={decorate,decoration={markings,mark=at position 0.5 with {\drawx}}}},
	kernelx/.style={semithick,shorten >=1pt,shorten <=1pt,->,
		postaction={decorate,decoration={markings,mark=at position 0.4 with {\drawx}}}},
	kernel1/.style={->,semithick,shorten >=1pt,shorten <=1pt,postaction={decorate,decoration={markings,mark=at position 0.45 with {\draw[-] (0,-0.1) -- (0,0.1);}}}},
	kernel2/.style={->,semithick,shorten >=1pt,shorten <=1pt,postaction={decorate,decoration={markings,mark=at position 0.45 with {\draw[-] (0.05,-0.1) -- (0.05,0.1);\draw[-] (-0.05,-0.1) -- (-0.05,0.1);}}}},
	kernelBig/.style={semithick,shorten >=1pt,shorten <=1pt,decorate, decoration={zigzag,amplitude=1.5pt,segment length = 3pt,pre length=2pt,post length=2pt}},
	kernelBigg/.style={thick,shorten >=1pt,shorten <=1pt,decorate, decoration={zigzag,amplitude=3.5pt,segment length = 7pt,pre length=2pt,post length=2pt}},
	kernelBigg1/.style={thick,shorten >=1pt,shorten <=1pt,decorate, decoration={saw,amplitude=3.5pt,segment length = 7pt,pre length=2pt,post length=2pt}},
	kernelBigg2/.style={thick,shorten >=1pt,shorten <=1pt,decorate, decoration={bumps,amplitude=3.5pt,segment length = 7pt,pre length=2pt,post length=2pt}},
	rho/.style={dotted,semithick,shorten >=1pt,shorten <=1pt},
	renorm/.style={shape=circle,fill=white,inner sep=1pt},
	labl/.style={shape=rectangle,fill=white,inner sep=1pt},
cumu2n/.style={inner sep=3pt},
cumu2/.style={draw=red!80,fill=red!40},
cumu3/.style={regular polygon, regular polygon sides=3,draw=red!80,rounded corners=3pt,fill=red!40,minimum size=5mm},
cumu4/.style={regular polygon, regular polygon sides=4,draw=red!80,rounded corners=3pt,fill=red!40,minimum size=7mm},
cumu5/.style={regular polygon, regular polygon sides=5,draw=red!80,rounded corners=3pt,fill=red!40,minimum size=7mm},
bcumu2n/.style={inner sep=3pt},
bcumu2/.style={draw=blue!80,fill=blue!40},
bcumu3/.style={regular polygon, regular polygon sides=3,draw=blue!80,rounded corners=3pt,fill=blue!40,minimum size=5mm},
bcumu4/.style={regular polygon, regular polygon sides=4,draw=blue!80,rounded corners=3pt,fill=blue!40,minimum size=7mm},
bcumu5/.style={regular polygon, regular polygon sides=5,draw=blue!80,rounded corners=3pt,fill=blue!40,minimum size=7mm},
	xi/.style={circle,fill=symbols!10,draw=symbols,inner sep=0pt,minimum size=1.2mm},
	xix/.style={crosscircle,fill=symbols!10,draw=symbols,inner sep=0pt,minimum size=1.2mm},
	xib/.style={circle,fill=symbols!10,draw=symbols,inner sep=0pt,minimum size=1.6mm},
	xibx/.style={crosscircle,fill=symbols!10,draw=symbols,inner sep=0pt,minimum size=1.6mm},
	not/.style={circle,fill=symbols,draw=symbols,inner sep=0pt,minimum size=0.5mm},
	>=stealth,
	}
\makeatletter
\def\DeclareSymbol#1#2#3{\expandafter\gdef\csname MH@symb@#1\endcsname{\tikz[baseline=#2,scale=0.13,draw=symbols]{#3}}\expandafter\gdef\csname MH@symb@#1s\endcsname{\scalebox{0.7}{\tikz[baseline=#2,scale=0.15,draw=symbols]{#3}}}}
\def\<#1>{\csname MH@symb@#1\endcsname}
\makeatother

\DeclareSymbol{Xi22}{0.5}{\draw (0,0) node[xi] {} -- (-1,1) node[not] {} -- (0,2) node[xi] {};}

\DeclareSymbol{Xi2}{-2}{\draw (0,-0.25) node[xi] {} -- (-1,1) node[xi] {};}
\DeclareSymbol{Xi3}{0}{\draw (0,0) node[xi] {} -- (-1,1) node[xi] {} -- (0,2) node[xi] {};}
\DeclareSymbol{Xi4}{2}{\draw (0,0) node[xi] {} -- (-1,1) node[xi] {} -- (0,2) node[xi] {} -- (-1,3) node[xi] {};}
\DeclareSymbol{Xi2X}{-2}{\draw (0,-0.25) node[xi] {} -- (-1,1) node[xix] {};}
\DeclareSymbol{XXi2}{-2}{\draw (0,-0.25) node[xix] {} -- (-1,1) node[xi] {};}

\DeclareSymbol{IXi2}{0}{\draw (0,-0.25) node[not] {} -- (-1,1) node[xi] {} -- (0,2) node[xi] {};}
\DeclareSymbol{IXi^2}{-1}{\draw (-1,1) node[xi] {} -- (0,0) node[not] {} -- (1,1) node[xi] {};}

\DeclareSymbol{XiX}{-2.8}{\node[xibx] {};}
\DeclareSymbol{Xi}{-2.8}{\node[xib] {};}
\DeclareSymbol{IXiX}{-1}{\draw (0,-0.25) node[not] {} -- (-1,1) node[xix] {};}

\DeclareSymbol{Xi3b}{-1}{\draw (-1,1) node[xi] {} -- (0,0) node[xi] {} -- (1,1) node[xi] {};}

\DeclareSymbol{IXi3}{2}{\draw (0,-0.25) node[not] {} -- (-1,1) node[xi] {} -- (0,2) node[xi] {} -- (-1,3) node[xi] {};}
\DeclareSymbol{IXi}{-2}{\draw (0,-0.25) node[not] {} -- (-1,1) node[xi] {};}
\DeclareSymbol{XiI}{-2}{\draw (0,-0.25) node[xi] {} -- (-1,1) node[not] {};}

\DeclareSymbol{Xi4b}{-1}{\draw(0,1.5) node[xi] {} -- (0,0); \draw (-1.2,1) node[xi] {} -- (0,0) node[xi] {} -- (1.2,1) node[xi] {};}
\DeclareSymbol{Xi4b'}{-1}{\draw(0,1.5) node[xi] {} -- (0,-0.2); \draw (-1,1) node[xi] {} -- (0,-0.2) node[not] {} -- (1,1) node[xi] {};}
\DeclareSymbol{Xi4c}{0}{\draw (0,1) -- (0.8,2.2) node[xi] {};\draw (0,-0.25) node[xi] {} -- (0,1) node[xi] {} -- (-0.8,2.2) node[xi] {};}
\DeclareSymbol{Xi4d}{-4.5}{\draw (0,-1.5) node[not] {} -- (0,0); \draw (-1,1) node[xi] {} -- (0,0) node[xi] {} -- (1,1) node[xi] {};}
\DeclareSymbol{Xi4e}{0}{\draw (0,2) node[xi] {} -- (-1,1) node[xi] {} -- (0,0) node[xi] {} -- (1,1) node[xi] {};}
\DeclareSymbol{Xi4e'}{0}{\draw (0,2) node[xi] {} -- (-1,1) node[xi] {} -- (0,-0.2) node[not] {} -- (1,1) node[xi] {};}

\def\sXi{\symbol{\Xi}}

\newtheorem{example}[lemma]{Example}
\newtheorem{assumption}[lemma]{Assumption}

\let\I\CI

\def\s{\mathfrak{s}}
\def\KK{\mathfrak{K}}

\let\eref\eqref

\def\Ren{\mathscr{R}}
\def\MM{\mathscr{M}}
\def\TT{\mathscr{T}}

\def\${|\!|\!|}
\def\Wick#1{\,\colon\!\! #1 \colon}
\def\E{\mathbf{E}}

\def\emptyset{\mathop{\centernot\ocircle}}

\def\diam{\mathrm{diam}}

\DeclareSymbol{X}{-2.4}{\node[dot] {};}
\DeclareSymbol{1}{0}{\draw[white] (-.4,0) -- (.4,0); \draw (0,0)  -- (0,1.2) node[dot] {};}
\DeclareSymbol{2}{0}{\draw (-0.6,1.6) node[dot] {} -- (0,0) -- (0.6,1.6) node[dot] {}; \node at (0,0) {};}
\DeclareSymbol{Kxi}{0}{\draw (0,2.2) node[dot] {} -- (0,-0.2); \node at (0,0) {};}
\DeclareSymbol{KKxi}{-3}{
	\draw (-0.5, 0.3) -- (0.4,-1.1) ;
	\draw  (-0.5, 0.3) -- (0.5,1.3) node[dot] {};
	\node at (0.3,0) {};}
\DeclareSymbol{K2}{0}{
	\draw (-0.9,2.2) node[dot] {} -- (0,1) -- (0.9,2.2) node[dot] {};
	\draw (0,1) -- (0,-0.2);
	\node at (0,0) {};
}

\DeclareSymbol{3}{0}{\draw (0,0) -- (0,1.2) node[dot] {}; \draw (-.7,1) node[dot] {} -- (0,0) -- (.7,1) node[dot] {};}

\DeclareSymbol{31}{-3}{\draw (0,0) -- (0,-1) -- (1,0) node[dot] {}; \draw (0,0) -- (0,1.2) node[dot] {}; \draw (-.7,1) node[dot] {} -- (0,0) -- (.7,1) node[dot] {};}
\DeclareSymbol{30}{-3}{\draw (0,0) -- (0,-1); \draw (0,0) -- (0,1.2) node[dot] {}; \draw (-.7,1) node[dot] {} -- (0,0) -- (.7,1) node[dot] {};}
\DeclareSymbol{32}{-3}{\draw (0,0) -- (0,-1) -- (1,0) node[dot] {}; \draw (0,0) -- (0,-1) -- (-1,0) node[dot] {}; \draw (0,0) -- (0,1.2) node[dot] {}; \draw (-.7,1) node[dot] {} -- (0,0) -- (.7,1) node[dot] {};}
\DeclareSymbol{22}{-3}{\draw (0,0.3) -- (0,-1) -- (1,0) node[dot] {}; \draw (0,0.3) -- (0,-1) -- (-1,0) node[dot] {};\draw (-.7,1) node[dot] {} -- (0,0.3) -- (.7,1) node[dot] {};}
\DeclareSymbol{20}{-3}{\draw (0,0) -- (0,-1);\draw (-.7,1) node[dot] {} -- (0,0) -- (.7,1) node[dot] {};}
\DeclareSymbol{12}{-3}{\draw (0,0.3) -- (0,-1) -- (1,0) node[dot] {}; \draw (0,0.3) -- (0,-1) -- (-1,0) node[dot] {};\draw (-.7,1) node[dot] {} -- (0,0.3);}
\DeclareSymbol{10}{-3}{\draw (0,0.3) -- (0,-1);\draw (-.7,1) node[dot] {} -- (0,0.3);}
\DeclareSymbol{21}{-3}{
	\draw (0,0.3) -- (0,-1) -- (1,0) node[dot] {};
	\draw (-.7,1) node[dot] {} -- (0,0.3) -- (.7,1) node[dot] {};
	\node at (0,0) {};}
\DeclareSymbol{21a}{-3}{
	\draw (-0.5, 0.3) -- (0,-1) -- (1,0) node[dot] {};
	\draw  (-0.5, 0.3) -- (0.5,1.3) node[dot] {};
	\node at (0.3,0) {};}
\DeclareSymbol{211}{1.5}{
	\draw (0,0) -- (-2,3) node[dot] {};
	\draw (0,0)  -- (1,1) node[dot] {};
	\draw (-0.67,1)  -- (.33,2) node[dot] {};
	\draw (-1.33,2)  -- (-0.33,3) node[dot] {};
	\node at (0,0) {};
}
\DeclareSymbol{4}{-4}{
	\draw (-1,0)  -- (0,-1.2) -- (1,0) ;
	\draw (-1.5,1) node[dot] {} -- (-1,0) -- (-0.5,1) node[dot] {};  
	\draw (1.5,1) node[dot] {} -- (1,0) -- (0.5,1) node[dot] {};
	\node at (0,0) {};
}

\begin{document}

\title{Moment bounds for SPDEs with non-Gaussian fields and application to the Wong-Zakai problem}
\author{Ajay Chandra$^1$ and Hao Shen$^2$}
\institute{University of Warwick, UK, \email{ajay.chandra@gmail.com}
%\and University of Warwick, UK, \email{M.Hairer@Warwick.ac.uk}
\and Columbia University, US, \email{pkushenhao@gmail.com}}

\maketitle

\begin{abstract}
Upon its inception the theory of regularity structures \cite{Regularity} allowed for the treatment for many semilinear perturbations of the stochastic heat equation driven by space-time white noise. When the driving noise is non-Gaussian the machinery of the theory can still be used but must be combined with an infinite number of stochastic estimates in order to compensate for the loss of hypercontractivity, as was done in \cite{CLTKPZ}. In this paper we obtain a more streamlined and automatic set of criteria implying these estimates which facilitates the treatment of some other problems including non-Gaussian noise such as some general phase coexistence models \cite{Phi4Weijun}, \cite{Shen-Xu} - as an example we prove here a generalization of the Wong-Zakai Theorem found in \cite{WongZakai}.
\end{abstract}

\setcounter{tocdepth}{2}
\tableofcontents

\section{Introduction}
In the paper \cite{WongZakai} the main focus was the convergence of smooth approximations $u_{\epsilon}$ to the solution of the SPDE
\begin{equ}\label{eq: Main SPDE}
\partial_{t} u = \partial_{x}^{2} u + H(u) + G(u)\xi.
\end{equ}
Here $u(t,x)$ is a function from $\R_{+} \times S^{1}$ to $\R$, $H,G: \R \rightarrow \R$ are respectively twice and five-times continuously differentiable and $\xi$ denotes space-time white noise. One can immediately obtain a solution $u$ to \eqref{eq: Main SPDE} by viewing it as an infinite dimensional It\^o integral equation in time. 

The fundamental obstacle to interpreting \eqref{eq: Main SPDE} without stochastic calculus is the irregularity of $\xi$. The smooth approximations $u_{\epsilon}$ satisfy the above equation with $\xi$ replaced by $\xi_{\eps} := \xi \ast \rho_{\eps}$ where $\rho_{\eps}$ is a mollifier converging to a space-time delta function as $\eps \downarrow 0$. More concretely the authors of \cite{WongZakai} set
\[
\rho_{\eps}(t,x) := \eps^{-3} \rho( \eps^{-2}t,\eps^{-1}x),
\]
where $\rho:\R^{2} \mapsto \R$ is an even, smooth, compactly supported function which integrates to $1$. Let $u_{\eps}$ denote the classical solution to the equation driven by $\xi_{\eps}$.

Unsurprisingly, the $u_{\eps}$ \emph{do not} converge to $u$ in general. One already sees this in finite dimensions where the Wong-Zakai Theorem 
\cite{WongZakai-1,WongZakai-2} (for more recent progress c.f. \cite{KaratzasRuf2013} and references therein)
states that smooth approximations to an SDE converge to the \emph{Stratonovich} solution to the SDE which in general differs from the It\^o solution. 
Of course this discrepancy can be cured by ``renormalizing'' the SDE by inserting a Stratonovich-It\^o correction term into the mollified SDEs. The main result of \cite{WongZakai} is the corresponding result for the SPDE setting.
\begin{theorem} \label{theo:WZgaussian}
(Hairer-Pardoux \cite{WongZakai}) 
Assume that $H$ and $G$ are of classes $\CC^2$ and $\CC^5$ respectively, both with bounded first derivatives. Let $\rho$ and $\xi_{\eps}$ be as above. Let $u$ denote the It\^o solution to \eqref{eq: Main SPDE} with some initial condition $u(0,\cdot) \in \CC(S^{1})$.
%let $u$ denote the solution to \eref{e:SPDE}, and 
Then there exist finite $\eps$ independent constants $C^{(1)}_{\rho},c_{\rho}^{(1)},$ and $c_{\rho}^{(2)}$ such that the following holds: If $u_\eps$ is the classical solution 
to the random PDE %\eref{e:SPDEapprox} 
\begin{equ}[e:SPDEapprox]
\d_t u_\eps = \d_x^2 u_\eps + \bar H(u_\eps) - \frac{C^{(1)}_{\rho}}{\eps} G'(u_\eps)G(u_\eps) + 
G(u_\eps)\,\xi_\eps\;,
\end{equ}
where
\begin{equ}[e:defbarH]
\bar H(u) \eqdef H(u) - c_\rho^{(1)} G'(u)^3G(u) - c_\rho^{(2)}G''(u)G'(u)G^2(u)\;.
\end{equ}
with initial condition $u_\eps(0,\cdot) = u(0,\cdot)$ then for any $T > 0$, one has
\begin{equ}
\lim_{\eps \to 0} \sup_{(t,x) \in [0,T] \times S^1} |u(t,x) - u_\eps(t,x)| = 0\;,
\end{equ}
in probability. For any $\alpha \in (0,{1\over 2})$ and $t > 0$, the 
restriction of $u_\eps$ to
$[t,T] \times S^1$ converges to $u$ in probability for the topology
of $\CC^{\alpha/2,\alpha}([t,T]\times S^1)$ as $\eps \downarrow 0$.
\end{theorem} 
We defer exact formulae for the renormalization constants; they can be explicitly written as integrals involving the heat kernel and $\rho$. The $\frac{1}{\eps}$ term in \eref{e:SPDEapprox} is exactly the It\^o-Stratonovich correction term which diverges as expected - there is no notion of infinite dimensional analog of Stratonovich integration. The definition of $\bar{H}(u)$ also involves two finite renormalizations which are chosen so that it is precisely the It\^o solution to which the $u_{\eps}$ are converging. In fact, along the proof \cite{WongZakai} obtains a natural notion of solution to \eqref{eq: Main SPDE} which is pathwise - an analogous situation as in \cite{MR2314753} (on rough paths) and \cite{MaxSam} (on evolution equations).

In \cite[Remark~1.7]{WongZakai} Hairer and Pardoux ask if an analogous statement can be proven if one replaces the mollified space-time white noise $\xi_{\eps}(t,x)$ with $\eps^{-3/2}\zeta(\eps^{-2}t,\eps^{-1}x)$ where $\zeta$ is a {\it non-Gaussian} random field which is supported on smooth functions and satisfies a central limit theorem.  
They conjectured that in addition to the renormalization seen in the Gaussian case one would see additional terms of order $\eps^{-\frac12}$. 

The question of \cite[Remark~1.7]{WongZakai} is our point of departure. 
%\hao{We use the word cumulant here and next page. I think it's OK to just refer to the KPZ paper for definition, and comment out the definition on page 21.}
%\hao{It seems we have mixing assumption on field, assumption on hyper-edges and assumption on cumulants. Should check if we have said they're equivalent or something.}
Let $\zeta$ be stationary, centered, generically non-Gaussian random field on $\R^{2}$ which is almost surely continuous and for which all cumulants\footnote{See Appendix~\ref{sec:app} for definition of cumulants.} exist and are exponentially decaying\footnote{See Definition \ref{def: exp decay condition on cumulants}.}. We also assume that $\zeta$ is normalized $\int \zeta(0)\zeta(z)\,dz=1$.
Let $\zeta^{(\eps)}$ be a random field on $\R\times [-\frac{1}{2\eps},\frac{1}{2\eps}]$ which is a periodization of $\zeta$ (see Remark~\ref{rem:periodization}).
We then set
\begin{equ} \label{e:def-zeta-eps}
\zeta_{\eps}(t,x) := \eps^{-3/2}\zeta^{(\eps)}(\eps^{-2}t,\eps^{-1}x).
\end{equ}
Our result is then the following, which is proved using the theory 
of regularity structures developed in \cite{Regularity}.
\begin{theorem} \label{theo:main}
Let $\zeta$ be as above and $H$ and $G$ be as in Theorem~\ref{theo:WZgaussian} and as before $u$ be the It\^o solution to \eqref{eq: Main SPDE} started from some initial condition $ u(0,\cdot) \in \CC(S^{1})$.
%let $u$ denote the solution to \eref{e:SPDE}, and 
Then there exist constants $\{C^{(i)}_{\zeta}\}_{i=1}^{3}$ and $\{c^{(j)}_{\zeta}\}_{j=1}^{4}$, all of which are independent of the parameter $\eps$, such that the following holds: Suppose that $u_\eps$ is the classical solution 
to the random PDE %\eref{e:SPDEapprox} 
\begin{equs}[e:SPDEapprox-NG]
	\d_t u_\eps = \d_x^2 u_\eps +  &\bar H(u_\eps) 
	- \frac{C^{(1)}_{\zeta}}{\eps} G'(u_\eps)G(u_\eps) \\
	&- \frac{C^{(2)}_{\zeta}}{\eps^{1/2}} G'(u_\eps)^2 G(u_\eps) - \frac{C^{(3)}_{\zeta}}{\eps^{1/2}} G''(u_\eps) G^2(u_\eps) 
	+ G(u_\eps)\,\zeta_\eps\;,
\end{equs}
where
\begin{equs}[e:defbarH-NG]
	\bar H(u) \eqdef H(u)  &
	- {1\over 6} c^{(1)}_{\zeta} G'''(u) G^3(u)
- c^{(2)}_{\zeta} G'(u)^3G(u)  \\
&- \Big( {1\over 2} c^{(3)}_{\zeta}
+c^{(4)}_{\zeta}\Big) G''(u) G'(u) G^2(u) \;
\end{equs}
and $u_\eps$ is started with the initial condition $u(0,\cdot) \in \CC(S^1)$.

Then for every $T>0$,  the family of random functions $u_{\eps}$
converges in law to the It\^o solution $u$  as $\eps \downarrow 0$,
with initial data $u_0$ in the space $\CC^{\alpha/2,\alpha}([0,T]\times S^1)$,
for any $\alpha\in(0,{1\over 2})$.
\end{theorem}
We now explictly specify the renormalization constants of appearing in the above theorem. We write
\begin{equs}  [e:def-consts]
C^{(1)}_{\zeta} &= C^{\<Xi2>}\qquad
C^{(2)}_{\zeta} = C^{\<Xi3>}\qquad
C^{(3)}_{\zeta} = C^{\<Xi3b>} \\
c^{(1)}_{\zeta} &= C^{\<Xi4b>} \qquad
c^{(2)}_{\zeta} = C_{1}^{\<Xi4>} +C_{2}^{\<Xi4>} +C_{3}^{\<Xi4>} \\
c^{(3)}_{\zeta} &= C^{\<Xi4c>}  \qquad
c^{(4)}_{\zeta} = C_{1}^{\<Xi4e>}+C_{2}^{\<Xi4e>}+C_{3}^{\<Xi4e>}
\end{equs}
where the new constants appearing above are defined by the diagrammatic formulae
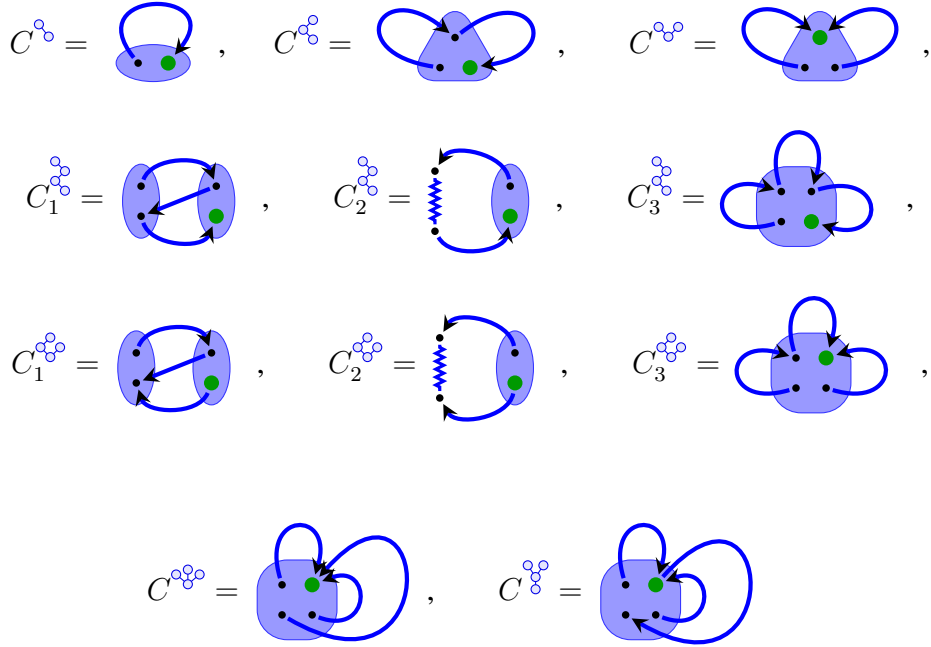
\begin{equ}
C^{\<Xi2>}= \!\!\!\!
\begin{tikzpicture}  [baseline=4]
\node[cumu2n]	(a)  at (0,0) {};	
\draw[bcumu2] (a) ellipse (14pt and 7pt);
\node[dot] (left) at (-.2,0) {};
\node[root] (right) at (.2,0) {};
\draw[Pkernel]   (left)  .. controls (-1,1) and (1,1) ..   (right);
\end{tikzpicture} 
\!\!\!\! \;,  \qquad
\!\!\!
C^{\<Xi3>} = \!\!\!\!
\begin{tikzpicture}  [baseline=0]
\node[bcumu3,minimum size=15mm,rounded corners=4mm] at (0,0) {};
\node[dot] (up) at (0,.2) {};
\node[dot] (left) at (-.2,-.2) {};
\node[root] (right) at (.2,-.2) {};
\draw[Pkernel]   (left)  .. controls (-1.5,0) and (-1,1) ..  (up);
\draw[Pkernel]  (up)	.. controls (1,1) and (1.5,0) ..  (right);
\end{tikzpicture} 
\!\!\!\! \;,   %c_\rho^{(\beta)} \;,
\qquad
C^{\<Xi3b>} = \!\!\!\!
\begin{tikzpicture}  [baseline=0]	
\node[bcumu3,minimum size=15mm,rounded corners=4mm] at (0,0) {};
\node[root] (up) at (0,.2) {};
\node[dot] (left) at (-.2,-.2) {};
\node[dot] (right) at (.2,-.2) {};
\draw[Pkernel]   (left)  .. controls (-1.5,0) and (-1,1) ..  (up);
\draw[Pkernel]  (right)	.. controls (1.5,0) and (1,1) ..  (up);
\end{tikzpicture} 
\!\!\!\! \;, 
%c_\rho^{(\gamma)}
\end{equ}
\begin{equ}
C_{1}^{\<Xi4>} = \;
\begin{tikzpicture}  [baseline=-2]	
\node[cumu2n]	(a)  at (0,0) {};	
\draw[bcumu2] (a) ellipse (7pt and 14pt);
\node[cumu2n]	(b)  at (1,0) {};	
\draw[bcumu2] (b) ellipse (7pt and 14pt);
\node[dot] (ul) at (0,.2) {};
\node[dot] (ur) at (1,.2) {};
\node[dot] (dl) at (0,-.2) {};
\node[root] (dr) at (1,-.2) {};
\draw[Pkernel,bend left =70]   (ul) to  (ur);
\draw[Pkernel]  (ur) to (dl);
\draw[Pkernel,bend right=70]  (dl) to  (dr);
\end{tikzpicture} 
%\to c_\rho^{(b)}
\;\;\;,
 \qquad
C_{2}^{\<Xi4>} = \;
\begin{tikzpicture}  [baseline=-2]	
%\node[cumu2n]	(a)  at (0,0) {};	
%\draw[bcumu2] (a) ellipse (7pt and 14pt);
\node[cumu2n]	(b)  at (1,0) {};	
\draw[bcumu2] (b) ellipse (7pt and 14pt);
\node[dot] (ul) at (0,.4) {};
\node[dot] (ur) at (1,.2) {};
\node[dot] (dl) at (0,-.4) {};
\node[root] (dr) at (1,-.2) {};
\draw[Pkernel,bend right=70]  (ur) to (ul);
%\draw[PkernelBig]   (ul) .. controls (-1,.5) and (-1,-.5) ..  (dl);
\draw[PkernelBig]   (ul)  to  (dl);
\draw[Pkernel,bend right=70]  (dl) to  (dr);
\end{tikzpicture} 
%\to c_\rho^{(b)}
\;\;\;,
 \qquad
C_{3}^{\<Xi4>} = \!\!
\begin{tikzpicture}  [baseline=0]	
\node[bcumu4,minimum size=15mm,rounded corners=4mm] at (0,0) {};
\node[dot] (ul) at (-.2,.2) {};
\node[dot] (ur) at (.2,.2) {};
\node[dot] (dl) at (-.2,-.2) {};
\node[root] (dr) at (.2,-.2) {};
\draw[Pkernel]   (dl) .. controls (-1.2,-.5) and (-1.2,.5) ..  (ul);
\draw[Pkernel]  (ul) .. controls (-.5,1.2) and (.5,1.2) ..  (ur);
\draw[Pkernel]  (ur) .. controls (1.2,.5) and (1.2,-.5) ..  (dr);
\end{tikzpicture}  
\;\;,
%\to c_\rho^{(b)}
\end{equ}
\begin{equ}
C_{1}^{\<Xi4e>} = \;
\begin{tikzpicture}  [baseline=-2]	
\node[cumu2n]	(a)  at (0,0) {};	
\draw[bcumu2] (a) ellipse (7pt and 14pt);
\node[cumu2n]	(b)  at (1,0) {};	
\draw[bcumu2] (b) ellipse (7pt and 14pt);
\node[dot] (ul) at (0,.2) {};
\node[dot] (ur) at (1,.2) {};
\node[dot] (dl) at (0,-.2) {};
\node[root] (dr) at (1,-.2) {};
\draw[Pkernel,bend left =70]   (ul) to  (ur);
\draw[Pkernel]  (ur) to (dl);
\draw[Pkernel,bend left=70]  (dr) to  (dl);
\end{tikzpicture} 
\;\;\;,
\qquad
C_{2}^{\<Xi4e>} = \;
\begin{tikzpicture}  [baseline=-2]	
\node[cumu2n]	(b)  at (1,0) {};	
\draw[bcumu2] (b) ellipse (7pt and 14pt);
\node[dot] (ul) at (0,.4) {};
\node[dot] (ur) at (1,.2) {};
\node[dot] (dl) at (0,-.4) {};
\node[root] (dr) at (1,-.2) {};
\draw[Pkernel,bend right=70]  (ur) to (ul);
%\draw[kernelBig]   (ul) .. controls (-1,.5) and (-1,-.5) ..  (dl);
\draw[PkernelBig]   (ul)  to  (dl);
\draw[Pkernel,bend left=70]  (dr) to  (dl);
\end{tikzpicture} 
\;\;\;,
\qquad
C_{3}^{\<Xi4e>} = \!\!
\begin{tikzpicture}  [baseline=0]	
\node[bcumu4,minimum size=15mm,rounded corners=4mm] at (0,0) {};
\node[dot] (ul) at (-.2,.2) {};
\node[root] (ur) at (.2,.2) {};
\node[dot] (dl) at (-.2,-.2) {};
\node[dot] (dr) at (.2,-.2) {};
\draw[Pkernel]   (dl) .. controls (-1.2,-.5) and (-1.2,.5) ..  (ul);
\draw[Pkernel]  (ul) .. controls (-.5,1.2) and (.5,1.2) ..  (ur);
\draw[Pkernel]  (dr) .. controls (1.2,-.5) and (1.2,.5) ..  (ur);
\end{tikzpicture}
\;\;,
%\to c_\rho^{(d)}
\end{equ}
\begin{equ}
C^{\<Xi4b>} = \;
\begin{tikzpicture}  [baseline=0]	
\node[bcumu4,minimum size=15mm,rounded corners=4mm] at (0,0) {};
\node[dot] (ul) at (-.2,.2) {};
\node[root] (ur) at (.2,.2) {};
\node[dot] (dl) at (-.2,-.2) {};
\node[dot] (dr) at (.2,-.2) {};
\draw[Pkernel]   (dl) .. controls (2,-1.5) and (1.8, 2) ..  (ur);
\draw[Pkernel]  (ul) .. controls (-.5,1.2) and (.5,1.2) ..  (ur);
\draw[Pkernel]  (dr) .. controls (1,-.5) and (1,.5) ..  (ur);
\end{tikzpicture} 
%\to c_\rho^{(d)}
\!\!\!\!\!\!  \;,
\qquad
C^{\<Xi4c>} = \;
\begin{tikzpicture}  [baseline=0]	
\node[bcumu4,minimum size=15mm,rounded corners=4mm] at (0,0) {};
\node[dot] (ul) at (-.2,.2) {};
\node[root] (ur) at (.2,.2) {};
\node[dot] (dl) at (-.2,-.2) {};
\node[dot] (dr) at (.2,-.2) {};
\draw[Pkernel]   (ur) .. controls (1.8, 2) and (2,-1.5) ..  (dl);
\draw[Pkernel]  (ul) .. controls (-.5,1.2) and (.5,1.2) ..  (ur);
\draw[Pkernel]  (dr) .. controls (1,-.5) and (1,.5) ..  (ur);
\end{tikzpicture}  
\end{equ}
These diagrams represent integrals of various kernels. The black vertices 
\tikz \node[dot] at (0,0) {};
 represent integrated variables in $\R^{2}$ while the single green vertex
\tikz[scale=0.7] \node[root] at (0,0) {}; 
 represents $0 \in \R^{2}$. 
A blue arrow \tikz[baseline=-4] \draw[Pkernel] (0,0) to (1,0);
 corresponds to the heat kernel $P$ evaluated at the difference of its terminal vertex and initial vertex. 
 The lightly blue shaded regions (e.g. 
  \begin{tikzpicture}[scale=0.4,baseline=-4] 
 \node[bcumu4] at (0,0) {}; 
 \node at (-0.35,0.35) [dot] {}; 
 \node at (-0.35,-0.35) [dot] {}; 
 \node at (0.35,0.35) [dot] {}; 
 \node at (0.35,-0.35) [dot] {}; 
 \end{tikzpicture}
 ) should be thought of as ``hyperedges'', they represent 
 a cumulant of $\zeta $ evaluated at the positions of the vertices within the region. 
(See Appendix~\ref{sec:app} for definition of the n-th cumulants, which will be written as $\fC_n$.)
%\hao{Sometimes written as $\kappa$. Should all be $\fC$.}
For example we have 
%\hao{Something seems wrong with the z in the equation and the sentence after it.}
\[
C^{\<Xi3>}
=\ 
\int_{\R^{2}} 
\int_{\R^{2}}
{\rm d}z_{1}\ 
{\rm d}z_{2}\ 
P(z_{2} - z_{1})
P(0 - z_{2})
\mathfrak{C}_{3}(z_{1},z_{2},0).
\]
%Note that by translation-invariance of the kernels appearing the integral on the RHS does not depend on $z$.
%
Finally, the notation 
	\tikz[baseline=-0.1cm] \draw[PkernelBig] (0,0) to (1,0); 
stands for a renormalized kernel. If the variables corresponding to its endpoints are $z_{1}$ and $z_{2}$ then it represents the kernel
\[
P(z_{1} - z_{2})\mathfrak{C}_{2}(z_{1} - z_{2}) - 
C^{\<Xi2>}\delta(z_{1} - z_{2}).
\]
\begin{remark} \label{rem:periodization}
In \eqref{e:def-zeta-eps} the field $\zeta^{(\eps)}$ on $\R\times [-\frac{1}{2\eps},\frac{1}{2\eps}]$ where $[-\frac{1}{2\eps},\frac{1}{2\eps}]$ is identified with the circle of length $\frac{1}{\eps}$ is said to be a periodization of $\zeta$, in the same sense as in \cite[Assumption~2.1]{CLTKPZ}, i.e. for every sufficiently small $\eps>0$, there exists a coupling of $\zeta^{(\eps)}$ and $\zeta$ such that for every $T > 0$ and every $\delta > 0$,
\begin{equ}  [e:coupling]
\sup_{|t| \le T\eps^{-2}}\sup_{|x| \le (1 - \delta)/(2\eps)} \lim_{\eps \to 0} \eps^{-3} \E |\zeta(t,x) - \zeta^{(\eps)}(t,x)|^2 = 0\;.
\end{equ}
As an example, 
let $\mu$ be a Poisson point process on $\R^2 \times [0,1]$ with uniform intensity measure,
let $\phi(t,x,a)$ be a continuous function bounded by $e^{-|t|-|x|}$, and set
\begin{equ}[e:defzeta]
\zeta(t,x) = \int_{\R^2 \times [0,1]} \phi(t-s, x-y, a) \mu(ds,dy,da)
	-   \mu(\phi) \;.
\end{equ}
Let $\mu^{(\eps)}$ be the periodic extension to $\R^2 \times [0,1]$ of 
a Poisson point process on $\R \times [-1/(2\eps),1/(2\eps)] \times [0,1]$ 
with uniform intensity measure,
and 
\[
\phi^{(\eps)}(t,x,a) \eqdef \phi(t,x,a) \tilde \phi_{\eps^{-1/4}} (t,x)
\]
 where $\tilde \phi_{\eps^{-1/4}} (t,x)$ is a continuous function 
that is equal to $1$ 
when $|x|<\eps^{-1/4}$ and $0$ when $|x|>2\eps^{-1/4}$,
 and let $\zeta^{(\eps)}$ be as 
in \eqref{e:defzeta}, with $\mu$ replaced by $\mu^{(\eps)}$ and $\phi$ replaced by $\phi^{(\eps)}$.
Then one can verify that for $\eps$ small enough, \eref{e:coupling} is satisfied, where 
the natural coupling between $\zeta$ and $\zeta^{(\eps)}$ is such that
$\zeta = \zeta^{(\eps)}$ on $\R \times [-\frac12 \eps^{-1}+\eps^{-\frac14}, \frac12 \eps^{-1}-\eps^{-\frac14}]$.
Note that the cumulants of $\zeta^{(\eps)}$ are allowed to have infinite range (but exponential decaying) in $t$.
\end{remark}
\begin{remark}
The same results as Corollaries~1.10, 1.11, 1.13 in \cite{WongZakai} (on local continuity of the solution with respect to the initial condition in pathwise sense, sharp regularity result for the solution etc.) can also be proved. Since the proofs follow along the same lines, we refrain from redoing them here.
\end{remark}
\subsection{Moment estimates for SPDE with non-Gaussian fields}
Much of our paper is spent developing a set of criterion for estimating moments of certain non-Gaussian random variables which will be written as $(\Pi_{0}\tau)(\phi_0^{\lambda})$ and defined in Section~\ref{sec:RS}. These random variables can be represented as rooted trees. Each of these random variables is a multilinear functional of the driving noise $\zeta_{\eps}$ and the tree describes how to write this functional as an integral - the edges correspond to kernels, at each leaf we have an occurrence of the driving noise, and all variables other than the root correspond to integrated space-time variables.
  
%\hao{We might need to discuss this subsection a bit: it's more or less fine for me to use the word "trees", and not mention chaos decomposition etc.}\ajay{Made some changes. I think this graphical notation is confusing to lots of people... } 
%\hao{Something missing here?}
%\ajay{Should be fine now}

Roughly speaking\footnote{See \cite[Section~5]{HairerPhi4} for a more in depth discussion.}, the $p$-th moments of this random variable can be written as a graphical sum where one takes $p$ copies of this tree, identifies their roots, and then sums over all the possible ways of grouping the leaves into cumulants. 
For Gaussian noise only second cumulants appear but in general we have higher order cumulants which we view as hyperedges (edges incident to more than two vertices). We refer to this graphical sum as a cumulant expansion.

Estimates on moments are needed to establish regularity/homogeneity of certain random space-time processes via a generalization of the Kolmogorov continuity criterion, the random variable we are estimating is the analog of an ``increment'' of the process with the size of the increment given by a parameter $\lambda$. Our primary goal is, for each $p$, a bound of the type
\begin{equation}\label{e: desired moment bound}
\left|
\E \left[\left(\Pi_{0}\tau\right)(\phi_{0}^{\lambda})^{p}\right]
\right| \lesssim \lambda^{|\tau|p}
\end{equation}
uniform in $\lambda \in (0,1]$ and $\eps$ small where $|\tau|$ is determined by the structure of the tree. Here, for any continuous function $\phi:\R^{2} \mapsto \R$, $z = (t,x) \in \R^{2}$, and $\lambda > 0$,
\[
\phi^{\lambda}_{z}(\bar{t},\bar{x}) \eqdef
\lambda^{-3} \phi(\lambda^{-2}( \bar{t} -t), \lambda^{-1}( \bar{x}-x)).
\]
For the bound \eqref{e: desired moment bound} we need to estimate, for each of the larger hypergraphs appearing in the cumulant expansion of the given moment, a complicated convolution of kernels and cumulants. 

The paper \cite{KPZJeremy} developed criteria for graphs which guarantee that they satisfy the desired upper bound. Moreover by hypercontractivity (e.g. \cite[Lemma~10.5]{Regularity}) one has
\[
\E \left[\left(\Pi_{0}\tau\right)(\phi_{0}^{\lambda})^{p}\right] 
\le
C_{p} 
\left|
\E \left[\left(\Pi_{0}\tau\right)(\phi_{0}^{\lambda})^{2}\right]
\right|^{p/2}.
\]
Thus in the Gaussian case one needs \eqref{e: desired moment bound} only for $p=2$ which is obtained by checking a small number of graphs. 

When hypercontractivity is not available establishing \eqref{e: desired moment bound} for every $p$ involves estimating infinitely many hypergraphs\footnote{One actually only needs \eqref{e: desired moment bound} bound for $p$ large enough to facilitate a Borel-Cantelli argument - but if $p > 10$ this is already daunting.}. This problem was tackled in \cite{CLTKPZ} where a criterion for individual trees was given which implies that any larger graph built by merging $p$ copies of this tree satisfies the necessary criterion in \cite{KPZJeremy}. However the criteria of \cite{CLTKPZ} requires one to do some manipulation on the trees and the larger graphs one got after the merging: (i) cumulants had to be replaced by collections of edges and good factors had to be distributed according to ``epsilon allocation rules'' and (ii) so called ``positively - renormalized'' edges had to be estimated by hand on a case by case basis leading to more trees to check. 

In this paper we provide streamlined criteria for these trees, proving the sufficiency of these criterion will also be easier. By working with hyperedges issue (i) is avoided - this makes proving very general results like \cite{Shen-Xu} much easier. Handling the positively-renormalized edges automatically deals with issue (ii) which makes the treatment of the Wong-Zakai problem much easier.

In Section~\ref{sec:RS} we fix our regularity structures and formulate the abstract fixed point problem for the Wong-Zakai equation in a space of modeled distributions. 
In Section~\ref{sec:Bounds} we prove Theorem~\ref{theo:ele-graph} which states that certain criteria on the graphs yield the desired moment bounds. 
In Section~\ref{sec:AppWZ} we apply the results obtained in Section~\ref{sec:Bounds} to the Wong-Zakai problem and prove Theorem~\ref{theo:main}.
%\hao{Will we be able to obtain some corollaries about parabolic homogenization? (See \cite{Bal2009convergence,Bal2010homogenization} \cite{PardouxPiatnitski2012} \cite{Hairer2013random}).
%For homogenization one certainly does not like Gaussianity assumption. \cite{Hairer2013random} already assumed non-Gaussianity.}
\subsection{Acknowledgements} AC and HS both thank Martin Hairer for numerous helpful conversations. AC also thanks Martin Hairer for financial support from his Leverhulme Trust award. 
HS was partially supported by the
NSF through DMS-1712684.
We also thank the reviewer for their close reading and helpful remarks.
\section{Regularity Structures} \label{sec:RS}
The moment estimates we prove will be used as input for the theory of regularity structures developed in \cite{Regularity} (see also \cite{KPZ}). This machinery allows us to go from these estimates to the construction of an actual solution to the SPDE in question along with convergence of regularized and renormalized solutions to this limiting solution. 

We refer readers looking for a detailed exposition to \cite{IntroRegularity},  \cite{FrizHairer}[Ch. 15], and \cite{CW}; our description of the theory will be quite brief. The most basic object in the theory is a \emph{regularity structure} which consists of a triple $(A,\mathcal{T}, \mathcal{G})$. The set $A \subset \R$ is a list of the possible homogeneities we allow in our expansions; it is assumed to be locally finite and bounded below. $\mathcal{T}$ is a graded vector space $\mathcal{T} = \oplus_{\alpha \in A} \mathcal{T}_{\alpha}$ where each $\mathcal{T}_{\alpha}$ is a Banach space with a distinguished basis. $\CG$ is a group of continuous linear transformations on $\mathcal{T}$ with the property that for all $\alpha \in A,\ \tau \in \mathcal{T}_{\alpha},$ and $\Gamma \in G$ one has $(\Gamma \tau - \tau) \in \oplus_{\beta < \alpha} \mathcal{T}_{\beta}$. A regularity structure is used to describe ``jets of abstract Taylor expansions''; the vector space $\mathcal{T}$ is the target space for the jets and the \emph{structure group} $\CG$ includes transformations on the target space corresponding to change of base-point operations.

\subsection{The Wong-Zakai regularity structure} \label{sec:WZsymbols}
The specific regularity structure we use for our Wong-Zakai type model is exactly the same as the one used in \cite{WongZakai};
 in particular $\mathcal{T}$ is spanned by a set of indeterminants $\tau \in \mathcal{W}$, each carrying a homogeneity $|\tau|$. We first define a a larger class of indeterminants and then take $\CW$ as an appropriate subset.

We start with the indeterminants $\mathbf{1}$, $X_{0}$, and $X_{1}$ which are the abstract counterparts of $1$, $t$, and $x$ - since our scaling is parabolic we set $|\mathbf{1}| = 0$, $|X_{0}| = 2$, and $|X_{1}| = 1$. Given a multi-index $k = (k_{0},k_{1}) \in \N^{2}$ we write $X^{k}$ as a shorthand for $X_{0}^{k_{0}}X_{1}^{k_{1}}$. For such $k$ we set $|k|_{\s} \eqdef 2k_{0} + k_{1}$ so $|X^{k}| = |k|_{\s}$. We also write $\bar{\CT}$ for the commutative algebra generated by $\mathbf{1}$, $X_{0}$, and $X_{1}$, the set of \emph{abstract polynomials}.

This set also carries a commutative, multiplicative structure (for which $\mathbf{1}$ will act as a unit). Given two indeterminants $\tau, \bar{\tau} $ we enforce that $|\tau \bar{\tau}| = |\tau| + |\bar{\tau}|$.  

We also introduce the indeterminant $\Xi$ which represents the driving noise $\xi$, we set $|\Xi| = -3/2 - \kappa$ where $\kappa$ is a fixed, small positive parameter \footnote{fixing $\kappa \in (0,1/10)$ suffices}. We also introduce an operator $\mathcal{I}(\cdot)$ on indeterminants, enforcing that $\mathcal{I}(X^{k})$ vanish for any multi-index $k$; if $\mathcal{I}(\tau) \not = 0$ we enforce $|\mathcal{I}(\tau)| = |\tau| + 2$. 
%\hao{About the upper truncation, $|\tau| \le -|\Xi|+\kappa$ suffices. We need $\frac{5}{2}$ only if we want to prove Corollary~1.13 in HP. Well, in fact we didn't mention any of those corollaries (1.10, 1.11, 1.13).}
%\ajay{We should discuss.}
%\hao{Added a remark 1.4.}

Define $\mathcal{U}$ to be the smallest collection of indeterminants which contains $\mathbf{1}$, $X_{0}$, and $X_{1}$ and that satisfies the conditions 
(i) $\tau \in \mathcal{U} \Rightarrow \mathcal{I}(\tau) \in \mathcal{U}$, (ii) $\tau \in \mathcal{U} \Rightarrow \mathcal{I}(\Xi \tau) \in \mathcal{U}$, and (iii) $\tau, \bar{\tau} \in \mathcal{U} \rightarrow \tau \bar{\tau} \in \mathcal{U}$. Finally we set
\begin{equ} \label{e:def-CW}
\mathcal{W} := \left\{ 
\tau \in \mathcal{U}
\cup \{\ \tau \Xi:\ \tau \in \mathcal{U} \}:\ 
|\tau| \le \frac{5}{2}
\right\}.
\end{equ}
$A$ is given by the set of homogeneities that appear in $\mathcal{W}$ which by \cite[Lemma~8.10]{Regularity} is bounded below and locally finite, also for each $\alpha \in A$ the vector space $\mathcal{T}_{\alpha}$ spanned by indeterminants of homogeneity $\alpha$ in $\CW$ is finite dimensional. 

Often we write the elements of $\CW$ as blue symbolic trees with $\sXi = \<Xi>$. Each occurrence of the abstract integration
map $\CI$ is then denoted by a downward 
straight line. The product $\tau$ and $\bar{\tau}$ is represented by attaching the trees for $\tau$ and $\bar{\tau}$ at the root.
For example, we have $\symbol{\Xi\I^2(\Xi)} = \<Xi22>$. Note that we never see an expression of the form $\Xi^{2}$ in $\CW$. 
We also use the shorthand
$\symbol{\Xi X_1} = \<XiX>$. 
The elements in $\CW$ with 
negative homogeneities are:
\begin{equs}[e:symbols]
|\<Xi> | = -{3\over 2} - \kappa \;, &\quad 
|\<Xi2>| =-1 - 2\kappa  \;, \quad
|\<Xi3>| = | \<Xi3b> | = -{1\over 2} - 3\kappa \;,\\
|\<XiX> | = -{1\over 2} - \kappa \;, \quad
| \<Xi4>| &=|\<Xi4b>|=| \<Xi4c>|= |\<Xi4e> |=- 4\kappa \;, \quad
| \<Xi2X>| =|\<XXi2> |=- 2\kappa\;.
\end{equs} 
%\subsection{The Wong-Zakai structure group}
Having defined the $\mathcal{T}$ of the Wong-Zakai regularity structure, now we turn to defining a structure group $\mathcal{G}$. To do this we introduce another set of indeterminants denoted $\CW_{+}$ and denote by $\CT_{+}$ the commutative algebra they generate. The construction of the structure group can be summarized as follows: there will be a \emph{single} ``abstract'' matrix of indeterminants from $\CT_{+}$ which acts on $\mathcal{T}$ - all the individual elements of $\mathcal{G}$ arise by specifying an appropriate map $f \in \CT_{+}^{\star}$ where $\CT_{+}^{\star}$ is the set of algebra homomorphisms from $\CT_{+}$ to $\R$. 

Following \cite{WongZakai} we set 
%\hao{$\CJ$ hasn't been introduced above? Also, is it true that we won't need the following half page in case we don't want to prove $M\in R$?} \ajay{Yes, need to say something about $\CJ$. Also, it is true we could remove this if we don't show $M \in R$ but I think it is fairly important that we \emph{do} show $M \in R$.....}
\[
\CW_{+}
\eqdef
\{ X_{0}, X_{1}\} \cup 
\left\{
\CJ_{k}(\tau):\ 
\tau \in \CW \setminus \bar{\CT},\ k \in \N^{2} \textnormal{ with } 
|k|_{\s} < |\tau| + 2
\right\}.
\]
Here the operators $\CJ_{k}(\cdot)$ on $\CW_{+}$ are analogous to $\CI[ \cdot ]$ on $\CW$. We use the convention that $\CJ_{k}(\tau) \eqdef 0$ if $|\tau| \le - 2 + |k|_{\s}$. 

The abstract matrix described earlier will be map $\Delta: \CT \mapsto \CT \otimes \CT_{+}$ which we now define recursively on $\CT$. The base cases are given by
\[
\Delta \mathbf{1} \eqdef \mathbf{1} \otimes \mathbf{1},\quad
\Delta \Xi \eqdef \Xi \otimes \mathbf{1},\quad
\textnormal{ and }\quad
\Delta X_{i} \eqdef X_{i} \otimes \mathbf{1} + \mathbf{1} \otimes X_{i} \;.
\] 
We then recursively set
\[
\Delta (\tau \bar{\tau})
\eqdef \Delta(\tau) \cdot \Delta(\tau) \quad
\textnormal{ and } \quad
\Delta \CI(\tau) \eqdef
(\CI \otimes {\rm Id})
\Delta \tau
+
\sum_{
l,k \in \N^{2}
}
\frac{X^{l}}{l!} \otimes \frac{X^{k}}{k!}
\CJ_{l+k}(\tau) .
\]
The product on the RHS of the first definition is component-wise. Also note that the sum in the second definition only has finitely many non-vanishing terms.

Given any $f \in \CT_{+}^{*}$ we define a linear transformation $\Gamma_{f}: \CT \mapsto \CT$ by setting
\begin{equation}\label{structure group def}
\Gamma_{f}\tau \eqdef 
({\rm Id} \otimes f) \Delta \tau.
\end{equation}
We then define $\mathcal{G}$ to be the set of all linear transformations of the above form. The only non-trivial thing is to check is that $\mathcal{G}$ forms a group. This is done by equipping $\CT_{+}$ with a Hopf-algebra structure for which $\Delta$ serves as a comodule coproduct - we refer the curious reader to \cite[Section~8]{Regularity}. 
\subsection{Admissible models and renormalized models}
\label{sec:Adm-model}
It is convenient to replace the heat kernel $P$ with a truncation $K:\R^{2} \setminus \{0\} \mapsto \R$ with the following properties: (i) $K(z)$ vanishes for $|z| > 1$, (ii) for  $|z| < 1/2$, $K(z) = P(z)$, (iii) $K$ is smooth on $\R^{2} \setminus \{0\}$, 
and (iv) 
$\int_{\R^{2}}{\rm d}z\ K(z)z^{k} = 0$ for any multi-index $k$
 with $|k|_{\s} < 3$.

The existence of such a $K$ is not hard to show, see \cite[Section~5]{Regularity}, and we consider it fixed for the rest of the paper.
We now have everything in place to define the set of \emph{admissible} 
models $\mathscr{M}$ on the Wong-Zakai regularity structure. 

A \emph{model} is a pair of maps $(\Pi,\Gamma)$ with $\Pi:\R^{2} \mapsto \mathcal{L}(\CT,\CS'(\R^{d}))$ which we write $z \mapsto \Pi_{z}$ and $\Gamma: \R^{2} \times \R^{2} \mapsto \CG$ which we write $(z,\bar{z}) \mapsto \Gamma_{z \bar{z}}$. Here $\mathcal{L}(\CT,\CS'(\R^{2}))$ is the space of linear maps from $\CT$ into the space of tempered distributions $\CS'(\R^{2})$. These maps are required to satisfy the algebraic conditions
\begin{equation}\label{eqs: alg req for model}
\Pi_{z} \Gamma_{z\bar{z}} = \Pi_{\bar{z}} \quad \textnormal{and} \quad
\Gamma_{z\bar{z}} \Gamma_{\bar{z} z'} = \Gamma_{z z'}
\textnormal{  for any } z,\bar{z},z' \in \R^{2}.
\end{equation} 

Let $\mathcal{B}$ be the set of all $\phi \in \CS'(\R^{2})$ supported on the ball of radius $1$ and with $|D^{k}\phi| \le 1$ where $D$ denoted differentiation and $k \in \N^{2}$ with $|k|_{\s} \le 2$. We also require that models satisfy the analytic bounds

%\hao{Probably mention in a few words why we will not care about the 2nd bound.
%Also, where will we need $\Gamma$, $\PPi$ etc. in this paper, besides the definition of admissible models? }
%\ajay{Added a remark about the second bound at end of this subsection. Also got rid of the $\PPi$ from the discussion. It seems good to mention $\Gamma$ because every single intro to regularity structures starts with $(\Pi,\Gamma)$}

\begin{equation}\label{eqs: analytic req for model}
\sup_{z \in \mathfrak{K}} 
\left| 
(\Pi_{z}\tau)\left(\phi^{\lambda}_{z}\right) 
\right| 
\lesssim \lambda^{|\tau|} \quad \textnormal{ and } 
\quad
\sup_{
\substack{
z,\bar{z} \in \mathfrak{K}\\
z \not = \bar{z} 
}
} 
\| \Gamma_{z \bar{z}} \tau\|_{\alpha} \lesssim |z-\bar{z}|^{|\tau| - \alpha}
\end{equation}
for each compact set $\mathfrak{K} \subset \R^{2}$, uniformly over $\phi \in \mathcal{B}$, $\lambda \in (0,1]$, $\tau \in \CW$ and $\alpha \in A$. In the second bound $\| \tau \|_{\alpha}$ denotes the $\CT_{\alpha}$-norm of the $\CT_{\alpha}$ component of $\tau$.

The set of models can be equipped with a family of pseudometrics indexed by compact sets $\mathfrak{K} \subseteq \R^{2}$ - for two models $Z = (\Pi,\Gamma)$ and $\bar{Z} = (\bar{\Pi},\bar{\Gamma})$ one sets $\$ Z ; \bar{Z}\$_{\mathfrak{K}}$ to the maximum of two optimal constants for each of the the bounds of \eqref{eqs: analytic req for model} where in the first and second bound the individual objects are replaced by differences $(\Pi_{z} - \bar{\Pi}_{z})$ and $(\Gamma_{z\bar{z}} - \bar{\Gamma}_{z\bar{z}})$, respectively. 
Together the pseudometrics $\$\cdot ; \cdot \$_{\mathfrak{K}}$ make the space of models a \emph{non-linear} metric-space \footnote{the non-linearity coming from the constraint \eqref{eqs: alg req for model}}. 

To formulate the condition of admissibility it is convenient to switch to parameterizing models by pairs of maps $(\Pi,f)$ where $\Pi$ is as before and $f: \R^{2} \mapsto \mathcal{T}_{+}^{\ast}$ which we write $z \mapsto f_{z}$. The correspondance between $f$'s and $\Gamma$'s is given by  $\Gamma_{z \bar{z}} = \Gamma_{f_{z}}^{-1} \Gamma_{f_{\bar{z}}}$, here we use the notation of \eqref{structure group def}. It is straightforward to verify that $(\Pi,\Gamma)$ constructed in such a way automatically satisfies condition \eqref{eqs: alg req for model}. 

The notion of \emph{admissibility} can then be stated as follows. 
\begin{definition}\label{def: admissible}
 A pair of maps $(\Pi,f)$ as above is said to be \emph{admissible} on $(\CT,\CG)$ if the following conditions hold for all $z, \bar{z} \in \R^{d}$, and for any multi-index $k$, 
%\hao{Usually it is just written like $(\Pi_{z}X^{k})(\bar{z}) = (\bar{z} - z)^{k} $. Is this a new formulation?}
%\ajay{This is how admissible models are defined in Hairer-Quastel.}
\[
(\Pi_{z}\mathbf{1})(\bar{z})\ = 1,
\quad
(\Pi_{z}X^{k}\tau)(\bar{z}) = (\bar{z} - z)^{k} (\Pi_{z}\tau)(\bar{z}), 
\quad
f_{z}(X^{k}) = (-z)^{k}
\]
and for every $\tau \in \CW$ with $\CI(\tau) \in \CW$,
\begin{equation*}
\begin{split}
f_{z}(\CJ_{k}\tau) 
=&
- \mathbf{1}_{\left\{ |k|_{\s} < |\tau| + 2 \right\}}
\times
\int_{\R^{2}}
D^{k}K(z-\bar{z})
(\Pi_{z}\tau)({\rm d}\bar{z}),\\ 
(\Pi_{z} \CI\tau)(\bar{z})
=&
\int_{\R^{2}}
K(\bar{z} - z')
(\Pi_{z}\tau)({\rm d}z')
+
\sum_{k}
\frac{(\bar{z} - z)^{k}}{k!}
f_{z}(\CJ_{k}\tau).
\end{split}
\end{equation*}
\end{definition}
We remark that if the pair $(\Pi,f)$  is admissible and $\Pi$ satisfies the first bound of \eqref{eqs: analytic req for model} then the $\Gamma$'s built from $f$ satisfy the second and $(\Pi,f)$ determines a model. We denote by $\MM$ the complete metric space of admissible models.
%\subsection{Canonical models and renormalization}

We now describe one way to lift a continuous space-time function $\psi$ to a corresponding admissible model $Z_{\psi} = (\Pi,f)$. The algebraic constraints placed on admissible models are quite strong - if we define 
\begin{equation}\label{def of canonical model}
(\Pi_{z}\Xi)(\bar{z}) \eqdef \psi(\bar{z})
\quad
\textnormal{  and  }
\quad
(\Pi_{z}\tau \bar{\tau} ) (\bar{z}) \eqdef (\Pi_{z}\tau)(\bar{z})
\cdot 
(\Pi_{z}\tau)(\bar{z})
\end{equation}
then one can use the identities of Definition \ref{def: admissible} to define the rest of the action of $(\Pi,f)$. We call the model $Z_{\psi}$ built this way the canonical model built from $\psi$ - we use the shorthand $Z_{\eps} = (\Pi^{\eps},f^{\eps}) \eqdef Z_{\zeta_{\eps}}$ where $\zeta_{\eps}$ is our rescaled random field. 

A defect of the family of models $Z_{\eps}$ is that they \emph{do not} converge to a limiting model in $\MM$ as $\eps \downarrow 0$, the key difficulties coming from symbols $\tau$ which correspond to products of insufficiently regular space-time processes. We will have to modify this family to get a new collection of renormalized models $\hat{Z}_{\eps} = (\hat{\Pi}^{\eps},\hat{f}^{\eps})$ - in general these new models will not satisfy the second identity of \eqref{def of canonical model} - as an example we will have $
\Pi_{z}^{\eps}(\CI[\Xi]\Xi)(\bar{z})
=
\Pi_{z}^{\eps}(\CI[\Xi])(\bar{z})
\Pi_{z}^{\eps}(\Xi)(\bar{z})
-
C^{\<Xi2>}\eps^{-1}$, without this subtraction the RHS would not converge to a meaningful object.  

It is a fairly non-trivial task to determine how to deform the product property of a canonical model and still be left with an admissible model. In the theory of regularity structures this type of deformation of the product property is encoded via the action of a linear map $M: \CT_{0} \mapsto \CT_{0}$ for an appropriate subset $\CT_{0} \subset \CT$. One then has the following theorem, which is combination of Prop. 8.36, Def. 8.41, Theorem 8.44 in \cite{Regularity} and Theorem B.1 of \cite{KPZJeremy}.
%\ajay{What else should I include here, a short subsection defining the Ito model?}
%\hao{We said something in the end. That might be enough.}
%\ajay{In Hairer-Quastel they show upper-triangularity of $\Delta^{M}$ implies the upper-triangularity of $\hat{\Delta}^{M}$ - added the reference.}
%\hao{I see.}
\begin{theorem}\label{thm: renormalization operator}
Let $\CT_{0} \subset \CT$ satisfy the properties that $\oplus_{\alpha \le 0} \CT_{\alpha} \subset \CT_{0}$ and
\[ 
\Delta \CT_{0} \subset 
\CT_{0} \otimes {\rm Alg}( \CT_{0} )
\]
where ${\rm Alg}( \CT_{0} )$ is the subalgebra of $\CT_{+}$ spanned by terms of the form $X^{k} \prod_{i} \CJ_{l_{i}}(\tau_{i})$ with $\tau_{i} \in \CT_{0}$. 

Let $M: \CT_{0} \mapsto \CT_{0}$ be a linear map that commutes with both the application of $\CI[ \cdot ]$ and multiplication by $X^{k}$. Then there exist a unique linear, multiplicative map $\hat{M}: \CT_{+} \mapsto \CT_{+}$ fixing abstract polynomials and a unique linear map $\Delta^{M}: \CT \mapsto \CT \otimes \CT_{+}$ such that one has
\begin{equation*} 
\hat{M} \CJ_{k} 
=
\mathcal{M}( \CJ_{k} \otimes 1)\Delta^{M} 
\quad \textnormal{ and } \quad
(1 \otimes \mathcal{M})(\Delta \otimes 1)\Delta^{M}
=
(M \otimes \hat{M}) \Delta.  %^{M}.
\end{equation*}
Suppose furthermore that the map $\Delta^{M}$ is \emph{upper triangular}, that is for every $\alpha \in A$ and $\bar{\tau} \in \CT_{\alpha}$ one has $\Delta^{M}\tau \in (\oplus_{\beta \ge \alpha} \CT_{\alpha} ) \otimes \CT_{+}$. Then if $(\Pi,f)$ is an admissible model then so is the renormalized model $(\Pi^{M},f^{M})$ given by
\[
\Pi^{M}_{z} \eqdef 
(\Pi_{z} \otimes f_{z}) \Delta^{M}\quad
\textnormal{ and }
\quad
f_{z}^{M} \eqdef f_{z} \circ \hat{M}.
\]
Furthemore, the family of $M$ satisfying the above properties form a group $\mathfrak{R}$ under composition. 
\end{theorem}
Later, we will prescribe renormalization maps $M^{\eps}$, sketch how one checks the upper-triangle condition for $\Delta^{M^{\eps}}$, and set $\hat{\Pi}^{\eps} \eqdef (\Pi^{\eps}_{z} \otimes f^{\eps}_{z}) \Delta^{M^{\eps}}$ and $\hat{f}^{\eps} = f_{z} \circ \hat{M^{\eps}}$.
Returning to our previous example, one will have $\Delta^{M^{\eps}}\<Xi2> = M^{\eps}\<Xi2>\otimes 1$ and $M^{\eps} \<Xi2> = \<Xi2> - C^{\<Xi2>}\eps^{-1}$. 
%{\bf If we really want to prove $M$ belongs to RG, we also need to define $\Delta$, $\CJ$, etc. Not sure if it worths it. Also, it seems we didn't define "models"?}
\subsection{Modeled distributions and abstract fixed point problem}
Given an admissible model $Z$, one can then start formulating abstract fixed point problems in spaces of modelled distributions $\CD^{\gamma,\eta}$.
\begin{definition}\label{def: modelled distributions} Given an admissible model $Z \in \mathcal{M}$ and $\gamma, \eta \in \R$ we define the space of modelled distributions $\mathcal{D}^{\gamma,\eta}$ to be the set of all functions $U: \R^{2} \mapsto \oplus_{\alpha < \gamma} \CT_{\alpha}$ such that for every compact set $\mathfrak{K} \subset \R^{2}$ one has
\begin{equation}\label{eq: modelled dist norm}
\|U\|_{\gamma,\eta}
\eqdef
\sup_{z \in \kappa}
\sup_{\alpha < \gamma}
\frac{\|U(z)\|_{\alpha}}
{|t|^{(\eta-\alpha)/2}}
+
\sup_{(z,\bar{z}) \in \mathfrak{K}^{(2)}}
\sup_{\alpha < \gamma}
\frac{\|U(z) - \Gamma_{z \bar{z}} U(\bar z)\|_{\alpha}}
{
(|t| \wedge |\bar{t}|)^{(\eta - \gamma)/2}
|z - \bar{z}|^{\gamma - \alpha}
}
<
\infty.
\end{equation}
Here $\mathfrak{K}^{(2)} \eqdef  \{ (z,\bar{z})  \in \mathfrak{K}^{2}:\  |z - \bar{z}| \le 1 \wedge \frac{1}{2} \sqrt{|t| \wedge |\bar{t}|} \} $.
\end{definition}
One of the main theorems in \cite{Regularity} says that there exists a reconstruction operator $\CR$ mapping the elements of $\CD^{\gamma,\eta}$ to 
 actual functions or distributions.
In the space $\CD^{\gamma,\eta}$ one can define the notions of multiplication and composition with smooth functions. It is also possible to construct a linear operator $\CP$ on this space which represents the space-time convolution by the heat kernel, namely one has $\CR \CP U = P*\CR U$ where $*$ is space-time convolution.
These constructions allow us to formulate and solve abstract fixed point problems  in the space $\CD^{\gamma,\eta}$, and then apply the operator $\CR$ on the abstract solution, which yields an actual function or distribution.
For instance, regarding the equation \eqref{eq: Main SPDE}, the abstract fixed point problem in the space $\CD^{\gamma,\eta}$ is formulated
 as follows:
\begin{equ}[e:FP]
U = \CP \bigl((\hat H(U) + \hat G(U)\Xi)\one_{t > 0}\bigr) + P u_0\;,
\end{equ}
where $P u_0$ is understood as naturally lifted to the abstract polynomials $\bar\CT$. We actually consider this fixed point problem in a subspace 
$\CD_\CU^{\gamma,\eta} \subset \CD^{\gamma,\eta}$
consisting of functions that take values in the span of $\CU$ rather than $\CW$ (see \eqref{e:def-CW}).

The functions $\hat H,\hat G$ appearing in \eqref{e:FP} are understood as follows.
Given $U \in \CD_\CU^{\gamma,\eta}$ with $U(z) = u(z) \one + \tilde U(z)$
where
$\tilde U(z)$ takes values in $\bigoplus_{\alpha > 0}\CT_\alpha$
and a smooth function $G\colon \R \to \R$,
we write
\begin{equ}[e:FHat]
\bigl(\hat G(U)\bigr)(z) = G(u(z))\one + \sum_{k\ge 1} {D^k G(u(z))\over k!} \tilde U(z)^{k}\;,
\end{equ}
with the understanding that the product between any number of terms such that their
homogeneity adds up to $\gamma$ or more vanishes. 
Another property we will use is that $\CP U -\CI U \in \bar\CT$, so any solution $U$ to
\eref{e:FP} satisfies
\begin{equ}[e:propFP]
U(z) - \CI\bigl(\hat H(U(z)) + \hat G(U(z))\Xi\bigr) \in \bar \CT\;,
\end{equ}
for all points $z = (t,x)$ with $t \in (0,T)$.

It follows from \eref{e:propFP} and \eref{e:FHat} that if we consider it as an element of $\CD^{\gamma,\eta}$ with $\gamma$
greater than, but sufficiently close to, ${3\over 2}$, then $U$ is of the form
\begin{equs}
U&=u\,\1+G(u)\,\<IXi>+ G'(u)G(u)\,\<IXi2>  + u'\,\symbol{X_1} \label{e:formU}\\
&\quad+ G'(u)^2 G(u)\,\<IXi3>   + {1\over 2} G''(u) G^2(u) \<Xi4d> + G'(u)u'\,\<IXiX>\;,
\end{equs}
for some functions $u$ and $u'$. The symbols appearing here are introduced in Subsection~\ref{sec:WZsymbols}.
In $\CD^{\gamma'}$ for
$\gamma' > 0$ sufficiently close to $0$ we have the identity
\begin{equs}[e:RHS]
\hat G(U)\sXi &= G(u)\,\<Xi>+ G'(u)G(u)\,\<Xi2> + G'(u)^2 G(u)\,\<Xi3> + G'(u)u'\,\<XiX> \\
&\quad  + {1\over 2} G''(u) G^2(u)\, \<Xi3b> + {1\over 6} G'''(u) G^3(u) \,\<Xi4b>
+ G'(u)^3 G(u)\, \<Xi4> \\
&\quad + {1\over 2} G''(u) G'(u) G^2(u)\, \<Xi4c> + G''(u) G'(u) G^2(u)\,\<Xi4e>\\
&\quad + G'(u)^2 u'\,\<Xi2X> + G''(u) G(u) u'\,\<XXi2>\;.
\end{equs}
This expansion is needed to derive the renormalized equations
\eqref{e:SPDEapprox-NG} and
\eqref{e:defbarH-NG}.
\section{Graphical Moment Bounds} \label{sec:Bounds}
As discussed earlier the moments we need to bound can be represented by sums of graphs with hyperedges representing higher cumulants of the non-Gaussian noise. In this section we prove Theorem~\ref{theo:ele-graph} which states that Assumption~\ref{ass:ele-graph} for a given tree implies the desired bounds for every graph that appears in the aforementioned sum for the moment of that tree. We do not specialize to \eqref{eq: Main SPDE} and instead work in the general setting of d-dimensional space $\R^d$ with fixed scaling $\s \in \N^{d}$.
\subsection{Assumptions on kernels associated to hyper-edges}
We start by recalling  the notion of labeled coalescence trees, 
which will be useful for both the definition of norms on the kernels that are functions of more than two variables, and the proof to Theorem~\ref{theo:big-bd}.
A labeled coalescence tree $(T,\ell)$ is a rooted binary tree with every inner node $v$ associated with a natural number $\ell_v$ which respects the partial order of the nodes, namely, $\ell_v \ge \ell_w$ whenever $v\ge w$ (i.e. $w$ belongs to the shortest path connecting $v$ to the root).

We denote by $\mathcal{T}_{n}$ the set of labeled coalescence trees with precisely $n$ leaves. Given $(T,\ell) \in \mathcal{T}_{n}$ we define $\mathcal{D}(T,\ell) \subset (\R^{d})^{n}$ to be the set of all tuples $(x_{1},\dots,x_{n})$ such that for any $1 \le i < j \le n$ one has $|x_{i} - x_{j}| \sim 2^{-\ell_{v_{i} \wedge v_{j}}}$ where $v_{1},\dots,v_{n}$ denote the leaves of $T$.
%\hao{Let's just use $\kappa$ for hyperedges in this section and use $\fC$ for cumulants in SPDE problems.}
%\ajay{Agree.}
\begin{definition}  \label{def:hyperhomo}
For any $\alpha \ge 0$ we define the following quantities and any function $\kappa_n$ of $n > 2$ space-time points, we define
\[
\|\kappa_n\|_{\alpha}
\eqdef\ 
\sup_{(T,\ell) \in \mathcal{T}_{n}}
\Big[
2^{-\alpha \ell_\rho}
\sup_{(x_{1},\dots,x_{n}) \in \mathcal{D}(T,\ell)}
|\kappa_{n}(x_{1},\dots,x_{n})|
\Big]\;.
\]
Here $\rho$ denotes the root of the tree $(T,\ell)$. 
\end{definition}
\begin{definition} \label{def:hyperhomo2}
Given any $\alpha \ge 0$ and $p \in \N$, and a function $\kappa_{2}:\R^{d} \setminus \{0\} \rightarrow \R$ set
\[
\|\kappa_2\|_{\alpha}
\eqdef
\Big(
\int_{\R^{d}} dx\ 
|\kappa_{2}(x)|
\Big)
\vee
\sup_{
\substack{
(T,\ell) \in \mathcal{T}_{2}\\
k \in \N^{d},\ |k|_{\s} \le p}
}
\Big[
2^{-(\alpha + |k|_{\s}) \ell_\rho}
\sup_{(x_{1},x_{2]}) \in \mathcal{D}(T,\ell)}
|\kappa_{2}(x_{1} - x_{2})|
\Big]\;.
\]
\end{definition}
\begin{remark}
Note that the definition of $| \cdot \|_{\alpha, p}$ given by Definition~\ref{def:hyperhomo2} dominates that given \cite[Definition~10.12]{Regularity} for kernels of homogeneity $\alpha$.
The extra integral included in Definition~\ref{def:hyperhomo} is needed in order to handle renormalizations that arise when $\alpha \ge |\s|$.
\end{remark}
%Since we are concerned with cumulants we will always set $\alpha = n |\mathfrak{s}|/2$ and write $\|\kappa_{n}\|$ instead of $\|\kappa_{n}\|_{\alpha}$. 
%\hao{$\eps$ and $\epsilon$ are different. Let's use $\eps$ throughout the paper.}
\begin{lemma}
Suppose the family of cumulants $\{\fC_{n}\}_{n \in \N}$ are exponentially decaying as in Definition~\ref{def: exp decay condition on cumulants}. Then for each $n \in \N$ one has
\begin{equ} \label{e:kappa-exp}
\| \fC_{n}^{(\eps)} \|_{\alpha}
\lesssim
\infty
\end{equ}
uniform in $\eps \in (0,1]$ when $\alpha = n/2 \times |\s|$, where the $\{ \fC_{n}^{(\eps)} \}_{n \in \N}$ are rescaled cumulants.
\end{lemma} 
\begin{proof}
Without loss of generality suppose \eqref{def: exp decay condition on cumulants} with $\theta = e^{-1}$. Fix $n$, and suppose we are given $\vec{z} = (z_{1},\dots,z_{n}) \in (\R^{d})^{n}$ - we can assume that $\diam(z_{1},\dots,z_{n}) = |z_{1} - z_{2}|$. It follows that we have
\[
\left| 
\fC_{n}^{(\eps)}(z_{1},\dots,z_{n}) 
\right|
\lesssim
\eps^{-n|\mathfrak{s}|/2}
e^{-\eps^{-1}|z_{1} - z_{2}|}.
\]
We want to show that if $\vec{z} \in \mathcal{D}(T,\ell)$ then the RHS above is bounded by some constant times $2^{n \ell(\rho) |\mathfrak{s}|/2}$. 
If $2^{-\ell(\rho)} < \eps$ then this is immediate (just bound the exponential factor
by one), so suppose instead that $2^{-\ell(\rho)} \ge \eps$. Then since $|z_{1} - z_{2}| \sim 2^{-\ell(\rho)}$ we have
\begin{equation*}
\left| 
\fC_{n}^{(\eps)}(z_{1},\dots,z_{n}) 
\right|
\lesssim
\eps^{-n|\mathfrak{s}|/2}
e^{-\eps^{-1}2^{-\ell(\rho)}}
\lesssim
\eps^{-n|\mathfrak{s}|/2} \times (\eps^{-1}2^{-\ell(\rho)})^{-n|\mathfrak{s}|/2}
%\times
%\sup_{t \ge 1} | e^{-t} t^{n|\mathfrak{s}|/2}|.
\end{equation*}
where we used the inequality 
$e^{-t} \lesssim t^{-n|\s|/2}$ for $t \ge 1$.
Thus the claim follows.
\end{proof}
%\begin{remark} 
%In general we may consider other types of hyper-edges, such as
%\begin{itemize}
%\item $\kappa(x,y,z)=|x-y|^{-1}|y-z|^{-1}|z-x|^{-1}$. Though it is symmetric it can not be bounded by the norm here (because the root $\rho$ depends on the largest distance while $\sup|\kappa|$ depends on the shortest edge). But it is natural to define its homogeneity as $3$.
%\item $\kappa(x,y,z)=|x-y|^{-2}|y-z|^{-1}|z-x|^{-1}$. I think if we define its homogeneity as just one single number we will lose the fact that it has asymmetric inner structure.
%\item Renormalised hyper-edges with $r<0$, such as in the case of 
%generalised KPZ. But for now we only assume $r=0$ for hyper-edges $e$.
%\item $e=\bar e\cup\{0\}$ where $\bar e$ is a usual edge and $r_{\bar e}>0$.
%\end{itemize}
%It seems that in general a hyper-edge's homogeneity has to be indexed by coalesence trees on its vertices, rather than a single value $\alpha$.
%\end{remark}
\subsection{The entire graph with hyper-edges} \label{sec:EntireGraph}
We now state a modified version of the bound on generalized convolutions found in \cite{KPZJeremy}. The difference here is that we allow for the presence of the hyperedges described above.
The proof of our version of the bound essentially follows in the same way as that found in \cite{KPZJeremy}, so instead of giving a full proof here we only list the ways in which the proof needs to be modified.

The basic setting for these proofs is encoding the key properties of our generalized convolution as a (decorated) finite graph $G=(\CV,\CE)$. $\CV$ as before is the vertex set, which includes a subset of distinguished vertices $\CV_{\star}$, one of which we call $0$.
The set $\CE$ can be decomposed as $\CE_2 \sqcup \CE_h$
($\sqcup$ meaning disjoint union)
 where $\CE_{2}$ is the set of normal directed edges (denoted by ordered pairs $(e_-,e_+)$ with $e_{-},e_{+} \in \CV$) and $\CE_{h}$ is the set of hyper-edges (subsets $ e \subset \CV$ with $|e| \ge 3$) \footnote{We sometimes use the term edge to refer to any element of $\CE$, not just the elements of $\CE_{2}$.}. 

We make further assumptions on the set $\CE$ which we list below.
\begin{itemize}
\item For any distinct $e_1,e_2\in\CE_h$ one has $e_1\cap e_2=\emptyset$.
\item For any  $e \in\CE_2$ one has $|e\cap \Big( \bigcup_{\bar e\in\CE_h} \bar e\Big)| \le 1$.
\item For all $e \in \CE_h$ one has $e \cap \CV_\star=\emptyset$.
\end{itemize}

The edges $e\in\CE$ are also decorated with labels $a_e,r_e$ where $a_{e} \in \R$ and $r_{e} \in \Z$. We now list assumptions we make on these labels.
\begin{itemize}
\item[(a)] For every $e\in\CE_h$, $a_e=|e||\s|/2$ and $r_e=0$. 
	%\footnote{Note that our definition here is such that the approximate Dirac distribution is represented by a normal edge, not a hyper-edge.}
\item[(b)] If $e\in\CE_2$, $\bar e\in\CE_h$, and $e\cap \bar e\neq\emptyset$,
then $r_e\le 0$ and $e\cap \bar e = \{ e_- \}$. 
	\footnote{This requirement allows us to ensure that we never need estimates on derivatives of the kernels associated to our hyperedges.}
\end{itemize}

The edges $e\in\CE_2$ are associated with kernels $K_e$ which are smooth functions on $\R^{d} \times \R^{d} \setminus \{0\}$ and satisfy $\|K_e\|_{a_e,p}<\infty$ for any $p>0$. 
The edges $e\in\CE_h$ are associated with functions $\kappa_e$ 
on $(\R^{d})^{|e|}$ with $\|\kappa_e\|_{a_e}<\infty$ where $\|\kappa_e\|_{a_e}$ is defined in Definition~\ref{def:hyperhomo}. 
We will write $\kappa_e (e)\eqdef \kappa_e(x_1,\dots,x_{|e|})$ if $e=\{x_1,\dots, x_{|e|}\}$.
For edges $e\in\CE_2$ one also has renormalized kernels $\hat K_e$ as follows.
If $r_e < 0$ then define the distribution
\begin{equ}[e:defRen]
\bigl(\Ren K_e\bigr)(\phi) = \int K_e(x) \Bigl(\phi(x) - \sum_{|k|_\s < |r_e|} {x^k \over k!} D^k\phi(0)\Bigr)\,dx + \sum_{|k|_\s < |r_e|} {I_{e,k} \over k!} D^k\phi(0)
\end{equ}  
where $\{I_{e,k}\}_{|k|_\s < |r_e|}$ is a collection of real numbers, and the distributional ``kernel'' $\hat K_e$ acts on smooth $\phi$ on $\R^d\times\R^d$ by
$
 \hat K_e (  \phi)
\eqdef\tfrac12\int  
\Ren K_e( \phi_{z})\,dz
$ where $\phi_{ z}(\bar z)\eqdef\phi((z+\bar z)/2,(z-\bar z)/2)$.
For $r_e \ge 0$, we define
\begin{equ}[e:KhatPos]
\hat K_e(x_{e_-}, x_{e_+}) = 
K_e(x_{e_+}-x_{e_-}) - \sum_{|j|_\s < r_e} {x_{e_+}^j \over j!} D^j K_e(-x_{e_-})\;.
\end{equ}

To lighten notation we assume that the $\kappa_{e}$ are always symmetric functions of their arguments.
%We remark that in our actual applications of the article the kernel $\kappa_{e}$ associated to a hyperedge $e$ will depend only on $|e|$, and in a such situation instead of writing $\kappa_{e}$ we write $\kappa_{n}$ where $n = |e|$.
%\hao{Oh do we really need so strong assumption?}
%\ajay{No, not at all - it is just for notational purposes.}
%\hao{Did we ever write $\kappa_n$ in this section? I can't find it.}
%\ajay{We did write $\kappa_{n}$ earlier but it seems not any more, I will removed the assumption}
%We assume $a_e=|e||\s|/2$ for all $e\in\CE_h$ for simplicity.
With these notations the key quantity of interest is as follows.
\begin{equ}\label{e:genconv}
\CI^G(\phi_\lambda,K,\kappa) 
\eqdef  \int_{(\R^d)^{\CV_0}} 
\prod_{e\in \CE_2}{\hat K}_e(x_{e_-}, x_{e_+})
\prod_{e\in \CE_h}\kappa_e(e)
\prod_{v\in \CV_\star\setminus \{0\}} \phi_\lambda(x_v) \,dx\;,
\end{equ}
where $x_v \in\R^d$ is the point corresponding to the vertex $v\in \CV$.

For any $\bar \CV \subset \CV$, the subsets 
%\begin{equs}
%\CE^\uparrow(\bar \CV) 
%	&= \{e  \in \CE^0 \,:\, e\cap \bar \CV = e_- \;\&\; r_e > 0\}\;,\\
%\CE^\downarrow(\bar \CV) 
%	&= \{e  \in \CE^0 \,:\, e\cap \bar \CV = e_+ \;\&\; r_e > 0\}\;,\\
%\CE_0(\bar \CV) &= \{e  \in \CE\,:\, e\cap \bar \CV = e\}\;,\\
%\CE(\bar \CV) &= \{e  \in \CE\,:\, e\cap \bar \CV \neq \emptyset\}\;.
%\end{equs}
$\CE^\uparrow(\bar \CV)$ and $\CE^\downarrow(\bar \CV)$ of $\CE$ are defined in the same way as in \cite{KPZJeremy}, 
namely,
$\CE^\uparrow(\bar \CV) 
	= \{e  \in \CE \,:\, e\cap \bar \CV = e_- \;\&\; r_e > 0\}$
and $\CE^\downarrow(\bar \CV) = \{e  \in \CE \,:\, e\cap \bar \CV = e_+ \;\&\; r_e > 0\}$,
in particular $\CE^\uparrow(\bar \CV), \CE^\downarrow(\bar \CV) \subset \CE_2$
by our assumption on the label $r_e$. We also set
\begin{equ} [e:E0E]
\CE_0(\bar \CV) 
\eqdef
\{e \in \CE: \ e \subset  \bar \CV \}\;,
\enskip
\textnormal{and}
\enskip
\CE(\bar \CV)
\eqdef
\{e \in \CE :\ e \cap \bar \CV \not = \emptyset\}\;.
\end{equ}
%Here, we use the notation $e = (e_-,e_+)$ for a directed edge.
%Note that a hyper-edge $e$ is contained in $\CE_0(\bar \CV)$ only if all the vertices of $e$ are in the subgraph $\bar\CV$.
We use the shorthands 
$r_e^+ = (r_e \vee 0)$ and $r_e^- = -(r_e \wedge 0)$. 
We now state our main assumptions on the labels $(a_{e},r_{e})_{e \in \CE}$.

\begin{assumption}\label{ass:big-graph}
The labelled graph $(\CV,\CE)$ satisfies the following properties.
\begin{itemize}\itemsep0em
\item[1.] For every edge $e \in \CE_2$, one has $a_e - r_e^- < |\s|$.
\item[2.] For every subset
$\bar \CV \subset \CV_0$ of cardinality at least $3$, one has 
\begin{equ} \label{e:assEdges} \tag{A.2}
\sum_{e \in \CE_0(\bar \CV)} a_e < |\s|\,(|\bar \CV| - 1)\;.
\end{equ}
\item[3.] For every subset
$\bar \CV \subset \CV$ containing $0$ and of cardinality at least $2$, one has 
\begin{equ} \label{e:assEdges1} \tag{A.3}
\sum_{e \in \CE_0(\bar \CV)} a_e 
	+ \sum_{e \in \CE^\uparrow(\bar \CV)}(a_e + r_e - 1)
	 - \sum_{e \in \CE^\downarrow(\bar \CV)} r_e < |\s|\,(|\bar \CV| - 1)\;.
\end{equ}
\item[4.] For every non-empty subset $\bar \CV \subset \CV\setminus\CV_\star$,
one has the bounds
\begin{equ} \label{e:assEdges2} \tag{A.4}
\sum_{e \in \CE(\bar \CV)\setminus \CE^\downarrow(\bar \CV)} a_e  
+ \sum_{e \in \CE^\uparrow(\bar \CV)} r_e
- \sum_{e \in \CE^\downarrow(\bar \CV)} (r_e-1) > |\s|\,|\bar \CV| \;.
\end{equ}
\end{itemize}
\end{assumption}

\begin{remark}
The second assumption above is automatic for $\bar \CV=e\in\CE_h$
since condition \eqref{e:assEdges} then asks that $|\s|\,|\bar \CV|/2<|\s|\,(|\bar \CV| - 1)$ which follows from $|e| = |\bar \CV|\ge 3$.
%
%We will label second cumulants by $(a,r)=(|\s|,-1)$.
\end{remark}

The main result in this subsection is the following.

\begin{theorem} \label{theo:big-bd}
Under Assumption~\ref{ass:big-graph}
we have the bound 
\begin{equ}\label{e:gen-conv-thm}
|\CI^G(\phi_\lambda,K,\kappa) | \lesssim \lambda^\alpha
	\prod_{e\in\CE_2} \|K_e\|_{a_e;p}
	\prod_{e\in\CE_h} \|\kappa_e\|_{a_e}
\end{equ}
where $\alpha=|\s||\CV \setminus \CV_\star|-\sum_{e\in\CE}a_e$, $\lambda\in (0,1]$, $p = \max\{ |r_{e}| : e \in \CE \} + 1$, and the proportionality constant only depends on the structure of $G$ and the labels $r_e$.
%{\bf In \cite{KPZJeremy} it is $\CV_\star$ instead of $\CV_0$, I think it's a typo..}
\end{theorem}
\subsubsection*{Multiscale expansion}
The proof of Theorem~\ref{theo:big-bd} is by a multiscale analysis implemented by a scale decomposition of all kernels.
For $e\in\CE_2$ the kernels $\hat K_e$ are decomposed into 
an infinite collection of kernels $\hat K_e^{(\mathbf n)}$ with $\mathbf n\in \N^3$ just as in \cite[Lemma~A.4,A.5]{KPZJeremy}.
We remark that as in \cite{KPZJeremy}, for an edge $e\in\CE_2$
without any kernel $K_e$ associated, the kernel $\hat K_e^{(\mathbf n)}$
is still defined and we set $(a_e,r_e)=(0,0)$.

We also implement a multiscale decomposition for the $\kappa_{e}$ as follows.

\begin{definition}\label{def: def-kappa-ne}
Given 
%$\mathbf n\in \CN_\lambda$ with $\CN_\lambda$
%defined in \cite{KPZJeremy} and
 $e=\{v_1,\dots,v_{|e|}\}\in\CE_h$, 
let me
\[
\mathbf n_e=\{\mathbf n_{i,j}\}_{1\le i <j\le |e|} \in (\N^3)^{{1\over 2}|e|(|e|-1)}
\]
where $\mathbf n_{i,j}=\mathbf n_{\{v_i,v_j\}}\in \N^3$. 
We define 
$\kappa_e^{(\mathbf n_e)}$
as follows. 
We set $\kappa_e^{(\mathbf n_e)}=0$ 
unless $\mathbf n_{i,j}=(m_{i,j},0,0)$ for every $i<j$; in the latter case, we set
\begin{equ} \label{e:def-kappa-ne}
\kappa_e^{(\mathbf n_e)} (x_{v_1},\dots,x_{v_{N}})
\eqdef \kappa_e (x_{v_{1}},\dots,x_{v_{N}}) \prod_{i\neq j}\Psi^{(m_{i,j})}(x_{v_i}-x_{v_j}) \;,
\end{equ}
where $N = |e|$, $\Psi^{(n)}$ is the cutoff function supported in the annulus of radius $\sim 2^{-n}$ with $k$-th derivative bounded by $2^{|k|_\s n}$,
and $\sum_n \Psi^{(n)}=1$.
\end{definition}

We can then decompose our generalized convolution as
\begin{equ}
\CI^G(\varphi_\lambda, K,\kappa) 
=\sum_{\mathbf n} 
%\CI^{\mathbf{n},G}_\lambda(K,\kappa)
\int_{(\R^d)^{\CV_0}}
\prod_{e\in\CE_2}\hat K_e^{(\mathbf n_e)}(x_{e_-}, x_{e_+})
\prod_{e\in\CE_h}\kappa_e^{(\mathbf n_e)}(e)
\prod_{v\in \CV_\star\setminus \{0\}} \phi_\lambda(x_v)  \,dx.
\end{equ}

%\begin{equ} \label{e:def-I}
%\CI_\lambda^G(K,\kappa) \eqdef
%\int_{(\R^d)^{\CV_0}}
%\prod_{e\in\CE_2}\hat K_e
%\prod_{e\in\CE_h}\kappa_e\,dx
%=
%\sum_{\mathbf n\in\CN_\lambda} \CI^{\mathbf{n},G}_\lambda(K,\kappa)
%\end{equ}
%where

%
%\begin{theorem}
%Under Assumption~\ref{ass:big-graph}
%we have the bound 
As in \cite{KPZJeremy}
for $\lambda \in (0,1]$, let
$ \CN_\lambda$ be the set of $\mathbf n$
such that $2^{-|\mathbf n_{e} | } \le \lambda $
for every $e=(0,v)$ with $v\in\CV_\star\setminus \{0\}$.
Since $\varphi_\lambda$ can be viewed as a kernel with $a=0$ and norm being $\lambda^{-|\s|}$, it suffices for the proof of 
Theorem~\ref{theo:big-bd} to show
\begin{equ}\label{e:gen-conv-suffice}
|\CI_\lambda^G(K,\kappa)| \lesssim \lambda^\alpha
	\prod_{e\in\CE_2} \|K_e\|_{a_e;p}
	\prod_{e\in\CE_h} \|\kappa_e\|_{a_e}
\end{equ}
where $\alpha=|\s||\CV_0|-\sum_{e\in\CE}a_e$ and 
$
\CI^G_\lambda(K,\kappa) 
=\sum_{\mathbf n\in\CN_\lambda} \CI^{\mathbf{n},G}_\lambda(K,\kappa)
$ with
%\begin{equ}
%\CI^G_\lambda(K,\kappa) 
%=\sum_{\mathbf n\in\CN_\lambda} 
%%\CI^{\mathbf{n},G}_\lambda(K,\kappa)
%\int_{(\R^d)^{\CV_0}}
%\prod_{e\in\CE_2}\hat K_e^{(\mathbf n_e)}
%\prod_{e\in\CE_h}\kappa_e^{(\mathbf n_e)}  \,dx.
%\end{equ}
\begin{equ}\label{e:def-I 2}
\CI^{\mathbf{n},G}_\lambda(K,\kappa)
\eqdef
\int_{(\R^d)^{\CV_0}}
\prod_{e\in\CE_2}\hat K_e^{(\mathbf n_e)}
\prod_{e\in\CE_h}\kappa_e^{(\mathbf n_e)} \,dx.
\end{equ}
%\end{theorem}

\subsubsection*{Multiscale clustering and coalescence trees}

As in \cite{KPZJeremy} one can associate a coalescence tree $(T,\ell)\in\CT(\CV)$
to any collection of vertex positions $\{x_{v}\}_{v \in \CV_{0}}$ with $x_{v} \in \R^d$ via a coalescing process.  For any two nodes $u,v$ of $T$, $u\wedge v$ denotes the common ancestor of $u,v$. For any edge $e\in\CE$ of the graph (may be hyperedge), $e_\uparrow$ is the common ancestor of all the leaves for the points in $e$, and  $e_\Uparrow$ is the immediate ancestor of $e_\uparrow$.

Given a labelled tree $(T,\ell)\in\CT(\CV)$ and a constant $c>0$, 
we define the set $\CN(T,\ell)$ of
functions $\mathbf n:\CV^2 \to \N^3$ as in \cite{KPZJeremy} with the additional constraint that for every $e\in\CE_h$ and every $\{v,w\} \subseteq e$ we enforce $\mathbf n_{(v,w)}=(m,0,0)$ with $|m-\ell_{v\wedge w}| \le c$.
If $\{v,w\}\notin \CE_2$ and $\{v,w\} \not \subset e$ for all $e\in\CE_h$,
then the set $\CN(T,\ell)$ imposes no requirement on $\mathbf n_{(v,w)}$.

\begin{lemma} \label{lem:sum-over-trees}
There exists $c >0 $ such that the following holds: for every $\mathbf n$ with the property that integral in \eqref{e:def-I 2} is non-vanishing,
there exists an element $(T,\ell) \in \CT(\CV)$ with $\mathbf n\in\CN(T,\ell)$.
\end{lemma}

\begin{proof}
The only difference in our setting versus that of \cite[Lemma~A.9]{KPZJeremy} are the additional constraints imposed by the requiring the support of
$\kappa^{(\mathbf n_e)}$ to be non-empty for every $e\in\CE_h$, however the argument remains exactly the same.
\end{proof}

Lemma~\ref{lem:sum-over-trees} allows us to bound $\CI_\lambda^G$ by a sum over labelled trees:
\begin{equ}   \label{e:sum-trees}
|\CI_\lambda^G(K,\kappa)| \lesssim
\sum_{(T,\ell)\in\CT_\lambda(\CV)} 
\sum_{\mathbf n\in\CN(T,\ell)} 
|\CI^{\mathbf{n},G}_\lambda(K,\kappa) | \;,
\end{equ}
where $\CT_\lambda(\CV) \subset \CT(\CV)$ is the set of coalescence trees such that $2^{-\ell_{v\wedge w}} \le \lambda$ for any $v,w\in\CV_\star$.
As in \cite{KPZJeremy}, when one wants to implement negative renormalizations to get a better (convergent) bound on the contribution from certain problematic labeled trees $(T,\ell) \in \CT_\lambda(\CV)$ the procedure is to replace the integrand appearing in \eqref{e:def-I 2} with a cleverly chosen function $\tilde \CK^{\mathbf{n}}(x)$ (see below) which satisfies $\supp \tilde \CK^{\mathbf{n}}\subset \CD(T,\ell)$ and integrates to the same value. We then write
\begin{equ}  
|\CI_\lambda^G(K,\kappa)| \lesssim
\sum_{(T,\ell)\in\CT_\lambda(\CV)} 
\Big(\prod_{v\in T^\circ} 2^{-\ell_v |\s|}\Big)
\sup_{\mathbf n\in\CN(T,\ell)} 
\sup_x |\tilde \CK^{{\mathbf n},T}(x)| \;,
\end{equ}
where $T^{\circ}$ is the set of inner nodes of $T$.

In \cite{KPZJeremy} the key criterion used to get the bound \eqref{e:gen-conv-suffice} is the following lemma. The distinguished node $v_{*}$ in this lemma will correspond to the internal node of $T$ which is first common ancestor of the leaves $\CV_{*}$.

\begin{lemma}(\cite[Lemma~A.10]{KPZJeremy})\label{multiclustering}
Given a coalescence tree $T$ with a fixed distinguished node $v_\star$,
and the set of labelings denoted by $\CN_\lambda(T)$ satisfying $2^{-\ell_{v_\star}}<\lambda$,
together with a function $\eta:T^{\circ} \to \R$, we write 
\begin{equ}
\CI_\lambda(\eta) = \sum_{n\in\CN_\lambda(T)} \prod_{v\in T^{\circ}}
	2^{-n_v\eta_v} \;.
\end{equ}
Furthermore suppose that the two following conditions on $\eta$ hold:
\begin{enumerate}
\item  For every $u \in T^{\circ}$, one has $\sum_{v \ge u} \eta(v) > 0 $.
\item For every $u \in T^{\circ}$ such that $u \le v_{*}$, one has 
$\sum_{v \not \ge u} \eta(v) < 0$.
\end{enumerate}
Then it follows that one has $\CI_\lambda(\eta)\lesssim \lambda^{|\eta|}$ uniform for $\lambda\in(0,1]$ where $|\eta| \eqdef \sum_{v \in T^{\circ}} \eta(v)$.
\end{lemma}

%\begin{proof}
Since Lemma~\ref{multiclustering} is a result only about coalescence trees, and has nothing to do with the graph $(\CV,\CE)$, we do not need to re-prove this lemma.
%\end{proof}

The goal is to show that Assumption \ref{ass:big-graph} implies that for any labeled tree $(T,\ell) \in \mathcal{T}_{\lambda}(\CV)$ we can find an $\eta: T^{\circ} \mapsto \R$ such that: (i) the bound
\[
|\CI^{\mathbf{n},G}_\lambda(K,\kappa)|
\lesssim
%\left(
\prod_{v \in T^{\circ}}
2^{-n_v\eta(v)}
\]
holds uniform in $\mathbf{n} \in \CN(T,\ell)$ and (ii) the above function $\eta$ satisfies the two conditions in Lemma~\ref{multiclustering}.

%{\bf The second way is not to decompose $\kappa_e$, as follows. (Now I think the first way is better (seems more systematic).)}
%
%Define 
%\begin{equ}
%\CI_\lambda^G(K,\kappa) \eqdef \int_{D_\lambda} 
%	\prod_{e\in\CE_2}\mathscr R K_e %(x_{e_-},x_{e_+}) 
%	\prod_{e\in\CE_h}\kappa_e \,dx
%\end{equ}
%where $D_\lambda \eqdef \{x\in (\R^d)^{\CV_0}: |x_v|\le \lambda,\forall v\in\CV_\star\}$. 

%\begin{remark}
%In \cite{KPZJeremy}, $\CI_\lambda^G$ is 
%defined as $\sum_{\mathbf n\in\CN_\lambda} \int_{(\R^d)^{\CV_0}}
%\prod_{e\in\CE}\hat K_e^{(\mathbf n_e)}$.
%We do not decompose $\kappa_e$ so it is not clear how to generalise it
%this way. 
%\end{remark}

\subsubsection*{Definition of $\eta$ and proof of the theorem}

As in \cite{KPZJeremy} let $A^{-} \subset \CE$ be the subset of edges $e$ with $r_e < 0$
such that
$e_\uparrow$ only has two descendants $e_- $ and $e_+$ in the tree $T$. 
Given any edge $e =(e_-,e_+)$ and any $r > 0$, we define an operator $\mathscr Y_e^r$ acting
on sufficiently smooth functions $V \colon \R^{\CV} \to \R$ by
\begin{equ}
\bigl(\mathscr Y_e^r V\bigr)(x) = V(x) - \sum_{|k|_\s < r} {(x_{e_+} - x_{e_-})^k \over k!} \bigl(D_{e_+}^k V\bigr)(P_e(x))\;,
\end{equ}
where $D_{e_+}$ is differentiation with respect to the coordinate $x_{e_+}$ and $\bigl(P_e(x)\bigr)_v=x_v$ if $v \neq e_+$ and $ \bigl(P_e(x)\bigr)_v=x_{e_-}$ if $v = e_+$ (in other words it turns $x_{e_+}$ to $x_{e_-}$).
% $P_e \colon \R^{\CV} \to\R^{\CV}$ is given by
%\begin{equ}
%\bigl(P_e(x)\bigr)_v = 
%\left\{\begin{array}{cl}
%	x_v & \text{if $v \neq e_+$,} \\
%	x_{e_-} & \text{otherwise.}
%\end{array}\right.
%\end{equ}

We replace the integrand in \eqref{e:sum-trees} by
\begin{equ}
\tilde\CK^{(\mathbf n)}(x) =
\Big( \mathscr Y_{e^{(k)}}^{r_{e^{(k)}}}
	\cdots \mathscr Y_{e^{(1)}}^{r_{e^{(1)}}} 
	\Big( \prod_{e\in\CE_2 \setminus A^-}K_e^{(\mathbf n_e)}
	\prod_{e\in\CE_h}\kappa_e^{(\mathbf n_e)} \Big)
\Big) (x)
\prod_{e\in A^-} K_e^{(\mathbf n_e)} (x_{e_-},x_{e_+})
\end{equ}
where $A^- =\{e^{(1)},\dots,e^{(k)}\}$.
By our assumption (see assumption (b) in the beginning of this subsection), if $e\in A^-$ intersects a hyper-edge, then the intersection is the single vertex $e_-$. Therefore the operator $\mathscr Y_e^r$ leaves $\kappa_e^{(\mathbf n_e)}$ unchanged, which is very important in the following proofs.

Define
\begin{equs} \label{e:def-eta}
\tilde \eta(v) & =|\s|+\sum_{e\in\CE}\tilde \eta_e(v) \;, \\
\tilde \eta_e(v)&= \eta_e(v) 
+ |r_e|\one_{e\in A^-} \big(\one_{e_\uparrow}(v) -\one_{e_\Uparrow}(v)\big)
\end{equs}
where
\begin{equs}
\eta_e(v) & =-a_e \one_{e_\uparrow}(v)
 +r_e \big(\one_{e_+\wedge 0}(v)-\one_{e_\uparrow}(v) \big) 
	\one_{r_e>0,e_+\wedge 0 >e_\uparrow} \\
& \quad+(1-r_e-a_e) \big(\one_{e_-\wedge 0}(v)-\one_{e_\uparrow}(v) \big) 
	\one_{r_e>0,e_-\wedge 0 >e_\uparrow} \;.
\end{equs}
Although the definition looks the same as in \cite{KPZJeremy},
the $e\in\CE$ here may be a hyper-edge, and in the case $e\in\CE_h$
we have $\tilde \eta_e(v) = -{1\over 2}|\s||e| \one_{e_\uparrow}(v)$.
In other words for a hyper-edge $e=\{v_1,\dots,v_n\}$,
we add a weight of $\alpha=|\s|n/2$ to the first common ancestor of $v_1,\dots,v_n$.

\begin{lemma}
The functions $\tilde\CK^{(\mathbf n)}$ satisfy the bound
\begin{equ}\label{e:sup-bound}
\Big(\prod_{v\in T^\circ} 2^{-\ell_v |\s|}\Big)
\sup_{\mathbf n\in\CN(T,\ell)} 
\sup_x |\tilde \CK^{(\mathbf n)}(x)| 
\lesssim \prod_{v\in T^\circ} 2^{-\ell_v \tilde\eta(v)}
\end{equ}
uniform in $\mathbf{n} \in\CN(T,\ell)$.
\end{lemma}

\begin{proof}
Since the operator $\mathscr Y_e^r$ leaves $\kappa_e^{(\mathbf n_e)}$ unchanged, the functions $\tilde\CK^{(\mathbf n)}$ can be factored as
\begin{equ} \label{e:BigKfactor}
\tilde\CK^{(\mathbf n)}=
\Big(\prod_{e\in A^-}K_e^{(\mathbf n_e)}
	\prod_{e\in\CE_h}\kappa_e^{(\mathbf n_e)}\Big)
\cdot \tilde\CK_1^{(\mathbf n)}
\end{equ}
and the last factor here is
\begin{equ}\label{e: K1}
\tilde\CK_1^{(\mathbf n)}(x) 
\eqdef
\sum_{k\in\partial A^-} 
\int_{(\mathbf R^d)^{A^-}}
%\mathbf D^k 
 \Big(\prod_{e \in A^-} D_{x_{e_+}}^{k_e}\Big)
	\Big(\prod_{e\notin A^-} K_e^{(\mathbf n_e)} (x|y)\Big)
	 \prod_{e\in A^-} \CQ_x^{k,e}(dy_e) \;
\end{equ}
We refer to \cite[Lemma~A.16]{KPZJeremy} for the precise definition
of the notations $\partial A^-$, $x|y$ and $\CQ_x^{k,e}(dy_e)$ appearing above,
but only remark that $\tilde\CK_1^{(\mathbf n)}$ and the first product 
 on the right hand side of \eqref{e:BigKfactor} can both be bounded as in \cite{KPZJeremy} by the right hand side of \eqref{e:sup-bound} if $\tilde\eta(v)$ were defined as in \eqref{e:def-eta} with $\sum_{e\in\CE}$ replaced by $\sum_{e\in\CE_2}$. By the multiplicative structure of  the second factor, it remains to show that for each $e\in\CE_h$
\begin{equ}
\sup_{x \in \R^{d}} |\kappa^{(\mathbf n_e)}_e(x)| \lesssim \prod_{v\in T^\circ} 2^{-\ell_v \tilde \eta_e(v)}
= 2^{-\ell_{e_\uparrow} \tilde \eta_e(e_\uparrow)}
= \big(2^{\ell_{e_\uparrow}}\big)^{{1\over 2}|\s||e|} \;.
\end{equ}
This follows immediately from Definitions~\ref{def:hyperhomo} and \ref{def: def-kappa-ne}.
Note that the cut-off functions $\Psi^{(m_{i,j})}$ in  \eqref{e:def-kappa-ne}
impose that the tree $(\bar T,\bar \ell)$ over the $|e|$ vertices of $e$
induced from the tree $(T,\ell)$ has precisely $e_\uparrow$ as its root.
\end{proof}

The proof of Theorem~\ref{theo:big-bd} is finished with the following lemma.

\begin{lemma}
For each $T \in \mathcal{T}(\CV)$ the map $\tilde{\eta}: T^{\circ} \mapsto \R$ as given in \eqref{e:def-eta} satisfies the two conditions of Lemma \ref{multiclustering} and one also has $|\tilde{\eta}| = |\s| \times (|\CV| - 1) - \sum_{e \in \CE(\CV)} a_{e}$.
\end{lemma}
\begin{proof}
The proof of \cite{KPZJeremy}[Lemma A.19] applies to our situation verbatim so we just give a sketch here. Fix $\nu \in T^{\circ}$, we write $L_{\nu} \subset \CV$ for the set of leaves which are descendants of $\nu$ in $T$. 

When one calculates $\sum_{v \ge \nu} \tilde{\eta}(v)$ the result takes three different forms. If $L_{\nu} = e$ for some $e$ with $r_{e} < 0$ then it takes the value $|\s| - a_{e} + r_{e}$. Otherwise, the value depends on whether  $ 0 \not \in L_{\nu}$ or $0 \in L_{\nu}$ - in the former case the value of the sum is given by the difference of the righthand and lefthand sides of \eqref{e:assEdges} of Assumption~\ref{ass:big-graph} while in the latter it is given by the difference of the righthand and lefthand sides of \eqref{e:assEdges1} of Assumption~\ref{ass:big-graph} - in both cases one takes $\bar{\CV} = L_{\nu}$.

On the other hand, if $\nu \le \nu_{*}$ then $\sum_{v \not \ge \nu} \tilde{\eta}(v)$ is given by the difference of the righthand and lefthand sides of \eqref{e:assEdges2} with $\bar{\CV} = \CV \setminus L_{\nu}$.
\end{proof}

%\hao{It corresponds to Lemma A.19 of Martin-Jeremy (one page proof). We may need to say something why that proof still works.}

\subsection{The elementary graphs}

In this subsection we will show that Assumption~\ref{ass:ele-graph} imposed on an ``elementary graph" will imply that Assumption~\ref{ass:big-graph} holds for any (much larger) graph constructed from Wick contracting the elementary graph. Basically the graphs showing in the pictures in  Section~\ref{sec:AppWZ}  are all examples of elementary graphs.

\begin{definition} \label{def:ele-graph}
An {\it elementary graph} $H$ is a connected graph  labelled with $(a_e,r_e)$ for each edge $e$ as in Subsection~\ref{sec:EntireGraph} above,
with only one {\it special} non-zero vertex $v_\star$, namely $H_\star = \{0,v_\star\}$. Additionally there is a unique {\it distinguished edge} $e_{\star}=\{0,v_\star\}$ attached to $0$ - other edges may also connect to $0$ but they are not called distinguished edges.
The set $H_0=H\setminus \{0\}$ can be decomposed as a union of two disjoint subsets of vertices
$H_0 = H_{ex} \sqcup H_{in}$ which we call external and internal vertices, respectively. 

We also enforce that $\deg(v)=1$ for every $v\in H_{ex}$ and $\deg(v)\ge 2$ for every $v\in H_{in}$. Any edge $e$ with $e \cap H_{ex} \not = \emptyset$ will be called an``external edge". Edges  $e \not = e_{\star}$ which are not external edges are called ``internal edges". 
The unique internal vertex connected to an external vertex $v$ will be denoted $i(v)$.  
We require that $v_\star \in H_{in}$, and that for every external edge $e$ one has $a_e = |\s| $. 
\footnote{In practice one can always attach a new edge with label $|\s|$ representing the Dirac function to an external vertex
to ensure this assumption holds.}
We also enforce that for all edges $e$ with $|e| = 2$ one has $a_{e} < 2|\s|$.
\end{definition}

We can construct graphs $\CV$  by {\it Wick contracting} several copies of $H$,
similarly as in \cite{CLTKPZ}, except that we now build hyper-edges over external vertices instead of identifying them. For a set $D$ we denote by $\CP(D)$ the set of partitions
of $D$.

\begin{definition} \label{def:WickContr}
Given a set $A$ and an integer $p>1$, let $\{ A^{(i)} \}_{i=1}^{p}$ be $p$ copies of $A$ and let $D$ be their $p$-fold disjoint union - that is $D=\sqcup_{i=1}^p A^{(i)}$.
For $\pi \in \CP(D)$ we say that 
\[
\pi\in \CP_w(D;A,p)\subset \CP(D)
\]
 if
%call $\pi$ a {\it $p$-fold Wick contraction} of $A$ if
for every $B\in \pi$, one cannot find $1 \le i \le p$ such that $B \subset A^{(i)}$. In other words, we enforce that every block of the partition $\pi$ must contain elements from at least two different copies of $A$. In particular, one must have $|B| > 1$. 
\end{definition}

\begin{definition} \label{def:make-big-graph}
Suppose that we are given an elementary graph $H$ and an integer $p>1$.
For $1\le i\le p$ let $H^{(i)}$ be a copy of the graph $H$.
%and a $p$-fold Wick contraction $\pi$ on $H_{ex}$.  
Suppose that we are also given a partition 
$\pi\in \CP_w(\sqcup_{i=1}^p H^{(i)}_{ex};H_{ex},p)$.
From $H$, $p$, and $\pi$ we will construct a labeled graph $\mathcal{G} = (\CU,\CE)$  which will be called a {\it $p$-fold Wick contraction} of $H$.

To define the vertex set $\CU$ we first start with $\sqcup_{i=1}^p  H^{(i)}$ and then identify all the $p$ copies of the distinguished vertex $0$. The edge set is given by
\[
\CE(\CU)
\eqdef
\big(\sqcup_{i=1}^p \CE_{0} (H^{(i)}) \big) \sqcup \CE_c(\CU) \;,
\qquad \textnormal{where  }
\CE_c(\CU) \eqdef \pi.
\]
As in the last section we have the decomposition $\CE(\CU)=\CE_2(\CU) \sqcup \CE_h(\CU)$.

Each edge $e\in\CE(\CU)$ is naturally associated with a label $(a_e,r_e)$,
which is $(|\s| ,-1)$ if $e\in\CE_c(\CU) \cap \CE_2(\CU)$, %and $|e|=2$, 
or $(|e||\s|/2,0)$ if $e\in\CE_c(\CU)\cap \CE_h(\CU)$, %and $|e|>2$, 
or otherwise 
inherits the label $(a_e,r_e)$ from $H$. We also set $\CU_{0} \eqdef \CU \setminus \{0\}$ and $\CU_{\star} \subset \CU$ to be given by the set $\sqcup_{i=1}^p H_\star^{(i)}$ with all the copies of $0$ identified.
\end{definition}

While we have defined enough structure to formulate Assumption \ref{ass:big-graph} for $\mathcal{G}$, 
it turns out that this labeled graph does \emph{not} quite satisfy this assumption in general; the second inequality will be violated whenever one has a block $B = \{u,v\}$ in $\pi$ of cardinality $2$. Pictorially one then has 
\begin{equation} \label{bad-chain}
\begin{tikzpicture}
\node at (0,0) (a) {$i(u)$};
\node at (1,1) (b) {$u$};
\node at (2,1) (c) {$v$};
\node at (3,0) (d) {$i(v)$};
\draw (a) -- (b) -- (c) -- (d);
\end{tikzpicture}
\end{equation}
where as before, for an external vertex $w \in H^{(i)}_{ex}$ we denote by $i(w)$ the unique element of $H^{(i)}_{in}$ which is connected to $w$ by an edge of $\CE_{0}(H^{(i)})$. In this scenario a subset $V \subset \{u,v,i(u),i(v)\}$ with $|V| > 2$ is called a \emph{ bad chain } for $\mathcal{G}$.

The outer two edges of \eqref{bad-chain} each carry a label $(|\s|,-1)$. While they are divergent by power counting we expect them to be integrable since they represent approximate identities. The solution is to perform the integration of vertices $u$ and $v$ \emph{before} we perform our multiscale analysis. Pictorially we replace \eqref{bad-chain} by 
\begin{equ}
\begin{tikzpicture}
\node at (0,0) (a) {$i(u)$};
\node at (2,0) (d) {$i(v)$};
\draw (a)  -- (d);
\end{tikzpicture}
\end{equ}
which carries a label $(|\s|,-1)$.

\begin{definition} \label{def:make-big-graph2}
Given a a partition
$\pi\in \CP_w(\sqcup_{i=1}^p H^{(i)}_{ex};H_{ex},p)$ we define a labeled graph 
$\mathcal{G}' = (\CV,\CE'(\CV))$ which we call a reduced $p$-fold Wick contraction of $H$.
Let $\mathcal{G} = (\CU,\CE(\CU))$ denote the corresponding non-reduced $p$-fold Wick contraction as in Definition~\ref{def:make-big-graph}.  
$\mathcal{G}'$ represents the reduced graph obtained after we have integrated out the following variables \footnote{In plain words, $\CU_{rem}$ is the set of vertices like $u$ and $v$ in \eqref{bad-chain} for which we want to remove from the vertex set.}
\[
\CU_{rem} = 
\bigsqcup_{
B \in \pi \cap \CE_{2}(\CU)
}
B.
\]
The vertex set is given by $\CV = \CU \setminus \CU_{rem}$, while the edge set is given by 
\begin{equation*}
\begin{split}
\CE'(\CV) &\eqdef 
\left\{
e \in \sqcup_{i=1}^p \CE_{0} (H^{(i)})
:\ e \cap \CU_{rem} = \emptyset
\right\} \sqcup 
\CE_{c}'(\CV),\\
\textnormal{where} \quad
\CE_{c}'(\CV) 
\eqdef&
\Big(\pi \setminus \CE_{2}(\CU) \Big)
\sqcup 
\{
\{i(u),i(v)\} :\ 
\{u,v\} \in \pi \cap \CE_{2}(\CU) 
\}.
\end{split}
\end{equation*}

As before, we write $\CE'(\CV) = \CE'_{2}(\CV) \sqcup \CE'_{h}(\CV)$. The edges of $\CE'_{c}(\CV) \cap \CE'_{2}(\CV)$ are given the label $(|s|,-1)$. Since $\CE'(\CV) \setminus [\CE'_{c}(\CV) \cap \CE'_{2}(\CV)] \subset \CE(\CU)$ we can just let all the other edges inherit their labeling from $\mathcal{G}$.

Finally we let $\CV_0 = \CV \setminus \{0\}$ and $\CV_\star = \CU_{\star}$ and also define a retraction map $\mathfrak{r}: \CU \mapsto \CV$ associated to $\mathcal{G}$ given by
\[
\mathfrak{r}(a)
\eqdef
\begin{cases}
i(a) & \textnormal{ if }a \in \CU_{rem},\\
a & \textnormal{ otherwise }.
\end{cases}
\]
\end{definition}

We now state our counterpart of Assumption~\ref{ass:big-graph} for the elementary graphs $H$. For various sets of edges we recall the notation \eqref{e:E0E}.

\begin{assumption}\label{ass:ele-graph} 
The labelled graph $(H,\CE)$ satisfies the following properties.
\begin{itemize}\itemsep0em
\item[1.] For every edge $e \in \CE_2$, one has $a_e - r_e^- < |\s|$ where $r_e^- = -(r_e \wedge 0)$.
\item[2.] For every subset
$\bar H \subset H_0$ of cardinality at least $3$, one has 
\begin{equ} \label{e:aEdges} \tag{H.2}
%\sum_{e \in \CE(\bar H^\star)} a_e < |\s|\,|\bar H|\;.
\sum_{e \in \CE_0(\bar H)} a_e 
< |\s|\,\Big(|\bar H_{in}| +{1\over 2} (|\bar H_{ex}| -1 -\one_{\bar H_{ex}=\emptyset}) \Big) \;.
\end{equ}
\item[3.] For every subset 
$\bar H \subset H$ containing $0$ and of cardinality at least $2$, one has 
\begin{equ}\label{e:aEdges1} \tag{H.3}
\sum_{e \in \CE_0(\bar H)} a_e 
	+ \sum_{e \in \CE^\uparrow( \bar H)}(a_e + r_e - 1)
	 - \sum_{e \in \CE^\downarrow(\bar H)} r_e 
< |\s|\,\Big(|\bar H_{in} | + \frac{1}{2}| \bar H_{ex}|\Big) \;.
\end{equ}
\item[4.] For every non-empty subset $\bar H \subset H\setminus H_\star$,
one has the bounds
\begin{equ} \label{e:aEdges2} \tag{H.4}
\sum_{e \in \CE(\bar H) \setminus \CE^\downarrow(\bar H)} a_e  
+ \sum_{e \in \CE^\uparrow(\bar H)} r_e
- \sum_{e \in \CE^\downarrow(\bar H)} (r_e-1) 
	> |\s|\,\Big(|\bar H_{in}| +{1\over 2} |\bar H_{ex}| \Big) \;.
\end{equ}
\end{itemize}
\end{assumption}

%\begin{remark} \ajay{Maybe we should remove this remark?}
%Note that if $\bar H_{ex}= \emptyset$,
%then \eqref{e:aEdges} of Assump.~\ref{ass:ele-graph}
%and \eqref{e:assEdges} of Assump.~\ref{ass:big-graph}
%%item 2. of Assumption~\ref{ass:ele-graph}
%%and item 2. of Assumption~\ref{ass:big-graph}
%are the same with $\bar H=\bar\CV$.
%In the case 
%$\bar H_{ex}\neq \emptyset$,
%these two conditions are also the same
%if we identify $\bar\CV=\bar H^\star  \eqdef \bar H \cup \{v\}$
%and the set of edges $\CE_0(\bar H^\star ) \eqdef \CE_0(\bar H)\cup \{e\}$
%where the hyper-edge $e \eqdef \{v\}\cup \bar H_{ex}$
%and $a_e = {|\s|\over 2} (|\bar H_{ex}|+1)$.
%%where $\bar H^\star$ is defined above by attaching a 
%%$|\bar H_{ex}|+1$ cumulant to $\bar H$.
%(The vertex $v$ here should be thought of as a
%noise vertex from another copy of $H$. Even if $|\bar H_{ex}|=1$
%and one needs to convolve as in Remark~\ref{rem:Convolve}
%the two conditions can be also shown as the same.)
%%\hao{Minor problem here: if $|\bar H_{ex}|=1$, we need to convolve out... Can postpone.}
%For a reader familiar with the dynamical $\Phi^4$ model another interesting test for  \eqref{e:aEdges} of Assump.~\ref{ass:ele-graph}
%%item 2. of Assumption~\ref{ass:ele-graph}
%is the $\Phi^4$ model in four space dimensions:
%in fact taking $\bar H$ to be any of the ternary trees
%(the entire tree except $0$)
%then \eqref{e:aEdges} becomes exactly an equality.
%\end{remark} 

Our goal is to show that Assumption \ref{ass:ele-graph} on an elementary graph $H$ will imply assumptions \ref{ass:big-graph} for any reduced $p$-fold Wick contraction built from $H$. It is more convenient to instead prove a weaker version of Assumption \ref{ass:big-graph} for the corresponding non-reduced Wick contraction. We have the following lemma which is a straightforward consequence of our definitions. 

\begin{lemma}\label{lemma:UV}
Let $H$ be an elementary graph,  $\mathcal{G} = (\CU,\CE(\CU))$ be a $p$-fold Wick contraction of $H$, and $\mathcal{G}' = (\CV,\CE'(\CV))$ be the corresponding reduced $p$-fold Wick contraction. Suppose that $\mathcal{G}$ satisfies items 1,3, and 4 of Assumption \ref{ass:big-graph} and that Eq.~(A.2) holds
for every subset
$\bar \CU \subset \CU_0$ of cardinality at least $3$ which is not a bad chain. 
Then the graph $\mathcal{G}'$ satisfies Assumption \ref{ass:big-graph}.
\end{lemma}
\begin{proof} The fact that $\mathcal{G}'$ will satisfy item $1$ is quite clear so we focus on the other items. 
The first key point is that for any of the conditions 2,3, and 4 of Assumption~\ref{ass:big-graph}, given an appropriate $\bar{\CV} \subset \CV$ the difference between the LHS and RHS's of the needed inequality remains the same if one replaces $\bar{\CV}, \CE'(\bar{\CV}), \CE_{0}'(\bar{\CV})$, and $\CE^{'\downarrow}(\bar{\CV})$ by $\bar{\CU} \eqdef \mathfrak{r}^{-1}(\bar \CV)$, $\CE(\bar{\CU}), \CE_{0}(\bar{\CU})$, and $\CE^{\downarrow}(\bar{\CU})$, respectively - one must have $ |\bar{\CU}| = 2n$ in which case making this switch increases both the LHS and RHS by  $2n|\s|$. 
The second point is that for $|\bar{\CV}| > 2$ the set $\mathfrak{r}^{-1}(\bar \CV)$ will not be a bad chain.
\end{proof}
We will denote the $p$ copies of $H$ by $H^{(1)},\dots,H^{(p)}$ and write $H^{(j)}_{in}$ and $H^{(j)}_{ex}$ for the sets of internal vertices and external vertices of these copies.
\begin{theorem} \label{theo:ele-graph}
Let $H$ be an elementary graph satisfying Assumption~\ref{ass:ele-graph}. Let $p \ge 2$, and let $(\CV,\CE)$ be a reduced $p$-fold Wick contraction of $H$ constructed as in Definition~\ref{def:make-big-graph2}.  Then $(\CV,\CE)$ satisfies Assumption~\ref{ass:big-graph}.
In particular, the conclusion \eqref{e:gen-conv-thm}  of Theorem~\ref{theo:big-bd} holds.
\end{theorem} 
\begin{proof}
Let $(\CU,\CE(\CU))$ be as in Definition~\ref{def:make-big-graph}, namely, the $p$-fold Wick contraction from which $(\CV,\CE'(\CV))$ is constructed.
It suffices to prove that $(\CU,\CE(\CU))$ satisfies the assumptions of Lemma~\ref{lemma:UV}.

Item 1 of the assumption on $(H,\CE)$, together with the definition of the labels for edges in $\CE_c(\CU)$, obviously implies item 1 of the assumption on $(\CU,\CE)$.

To prove item $2$ let
$\bar \CU\subset \CU_0$ be vertex set of cardinality at least $3$ which is not a bad chain. We aim to show \eqref{e:assEdges}. 
We first show that it is sufficient to treat the case where $\bar\CU$ is {\it connected}. %via $\CE_{0}(\CV)$. 
First we claim it suffices to treat the case where no connected component of $\bar{\CU}$ is a bad chain - if there is such at least one such connected component it suffices to prove \eqref{e:assEdges} for the union of all the other connected components of $\bar \CU$.

If $\bar \CU$ has connected components of size $1$ then it suffices to check \eqref{e:assEdges} for the smaller vertex set where one drops these components; and the same holds for components of size $2$ 
(here one uses the assumption $|a_{e}| < 2 |\s|$ in Definition~\ref{def:ele-graph}). If all the components of $\bar \CU$ have cardinality at least $3$, then if 
\eqref{e:assEdges} holds for each of these components then summing up these bounds yields an even stronger bound for $\bar\CU$. This covers all the disconnected cases so we assume that $\bar \CU$ is {\it connected}.

Let $\bar H^{(i)}\eqdef \bar\CU\cap H^{(i)}$. 
Let $J_1,J_2,J_{\ge 3}$ be three disjoint subsets of $\{1,\dots,p\}$, such that
$|\bar H^{(i)}| =1$ for each $i\in J_1$, and $|\bar H^{(i)}| =2$ for each $i\in J_2$,
and $|\bar H^{(i)}| \ge 3$ for each $i\in J_{\ge 3}$. The LHS of \eqref{e:assEdges} is then
\begin{equs} [e:Proof-item2]
\sum_{e \in \CE_{0}(\bar \CU)} & a_{e}  
= \sum_{i\in J_1} \sum_{e \in \CE_{0}(\bar H^{(i)})} \!\!\!\! a_{e} 
+ \sum_{i\in J_2} \sum_{e \in \CE_{0}(\bar H^{(i)})} \!\!\!\! a_{e} 
+ \sum_{i\in J_{\ge 3}} \sum_{e \in \CE_{0}(\bar H^{(i)})} \!\! a_{e} 
+ \!\!\!\! \!\!\!\!  \sum_{e \in \CE_{0}(\bar\CU) \cap \CE_{c}(\CU)} \!\!\!\!\!\! a_{e} \\
&\!\!\!\!
\le |J_2| |\s| + \sum_{i\in J_{\ge 3}} |\s|\,\left[ |\bar H_{in}^{(i)}| 
	+{1\over 2} (|\bar H_{ex}^{(i)}| -1-\one_{\bar H_{ex}^{(i)}=\emptyset}) \right]
+ \frac{|\s|}{2} \!\!\!\!\sum_{i\in J_1\cup J_2\cup J_{\ge 3}} \!\!\!\!\! |\bar H_{ex}^{(i)}|
\end{equs}
For going to the bottom lines we use the following reasoning for each of the terms on the first line: (i) the sum over $J_1$ is obviously zero, (ii) since $|\bar\CU|\ge 3$ and is assumed to be connected, the edges involved in the $J_2$-summation are all external with labels $|\s|$, (iii) one can apply \eqref{e:aEdges} for the sum over $J_{\ge 3}$, and (iv) one has
\begin{equ}\label{cumulant bound}
\sum_{e \in \CE_{0}(\bar\CU) \cap \CE_{c}(\CU)} a_{e}
=
|\s|/2
\times
\sum_{e \in \CE_{0}(\bar\CU) \cap \CE_{c}(\CU)} |e|
\le
|\s|/2
\times
\sum_{i\in J_1\cup J_2\cup J_{\ge 3}}\  |\bar H_{ex}^{(i)}|.
\end{equ} 

Observe that if $J_{\ge 3}\neq \emptyset$ then by \eqref{e:aEdges} the inequality in \eqref{e:Proof-item2} is actually strict.

We now deal with the case where $J_1=J_2=\emptyset$ and $|J_{\ge 3}|=n \ge 1$. If $J_{\ge 3} = \{i\}$ then the last term on the first line of \eqref{e:Proof-item2} must vanish so
\begin{equ}\label{e:Proof-item2a}
\sum_{e \in \CE_{0}(\bar \CU)} a_{e}
<
|\s|\,\big(|\bar H_{in}^{(i)}| 
	+{1\over 2} (|\bar H_{ex}^{(i)}| -1-\one_{\bar H_{ex}^{(i)}=\emptyset}) \big)
\end{equ}
with the RHS being bounded above by $|\s|(|\bar{\CU}|-1)$, as desired. If $n > 1$ we must have $\bar{H}^{(i)}_{ex} \not = \emptyset$ for all $i \in J_{\ge 3}$ and our claim follows from the fact that
\begin{equ}\label{e:Proof - item2b}
\sum_{e \in \CE_{0}(\bar \CU)} a_{e}
<
|\s| \times \sum_{i \in J_{\ge 3}}
\left[
|\bar H_{in}^{(i)}| + 
(1/2 + 1/2)
|\bar H_{ex}^{(i)}| 
- \frac{1}{2} 
\right]
= |\s|(|\bar{\CU}| - n/2).
\end{equ}

Note that in all remaining cases one must have $\bar H_{ex}^{(i)} \not = \emptyset$ for $i \in J_{1} \sqcup J_{2} \sqcup J_{\ge 3}$. 
The case when $|J_{\ge 3}| \ge 1$ and $J_{1} \sqcup J_{2} \not = \emptyset$  follows similarly to the two cases treated above: the upper bounds in \eqref{e:Proof-item2a} and \eqref{e:Proof - item2b} will apply if we increase them by $(\frac{1}{2}|J_{1}| + \frac{3}{2}|J_{2}|) \times |\s|$ but the quantity $|\s|  (|\bar{\CU}| - 1)$ we compare them against goes up by $(|J_{1}| + 2|J_{2}|) \times |\s|$.

Henceforth we assume $J_{\ge 3} = \emptyset$. Suppose that $(|J_{1}|,|J_{2}|) = (1,1)$ or $(0,2)$. Since $\bar{\CU}$ is not a bad chain it must be the case that
$\CE_{0}(\bar\CU) \cap \CE_{c}(\CU) = \emptyset$. Then by \eqref{e:Proof-item2} one has $\sum_{e \in \CE_{0}(\bar \CU)} a_{e} \le |J_{2}| \times |\s|$ which is strictly smaller than 
$|\s| \times (|J_{1}| + 2|J_{2}| - 1)$.

In the remaining scenarios $J_{\ge 3} = \emptyset$ and $(|J_{1}|,|J_{2}|) \not = (1,1)$ or $(0,2)$ - \eqref{e:assEdges} then follows by observing that
\[
\sum_{e \in \CE_{0}(\bar \CU)} 
a_{e}
\le
|\s|
\times
\left( \frac{3}{2} \times |J_{2}| + \frac{1}{2}|J_{1}| \right)
<
(2 |J_{2}| + |J_{1}| - 1 ) \times |\s|
=|\s|
(|\bar{\CU}| - 1).
\] 

We now turn to proving that 
%item 3 of Assumption \ref{ass:ele-graph} implies item 3 of Assumption \ref{ass:big-graph}. 
 \eqref{e:aEdges1}  implies  \eqref{e:assEdges1}. 
%\hao{Remark that $\CE_{c}( \bar \CV) $ only includes cumulants between half-graphs, but doesn't include any cumulant within one half-graph?}
%\ajay{Fixed}
%For any subset of vertices $\bar \CV \subset \CV$ we define $\CE_{c}( \bar \CV) = \CE_{0}(\bar \CV) \cap \CE_c(\CV)$ to be the set of ``internal'' cumulant edges which touch different elementary graphs, that is all hyperedges $e$ with $e \subset \bar \CV$ but also satisfying the condition that $ \not \subset H^{(j)}$ for any $1 \le j \le n$.
Let $\bar\CU \subset \CV$ be of cardinality at least two with $\bar \CU \ni 0$. Let $J$ be the set of $j \in \{1,\dots,p\}$ with the property that $|\bar \CU \cap H^{(j)}| \ge 2$.  We immediately have 
%
%We can write
%\[
%\bar \CV = \{0\} \cup \Big( \bigcup_{i=1}^{n} \bar{H}^{(i)}_{in} \cup \bar{H}^{(i)}_{ex} \Big).
%\] 
\begin{equation}\label{sum positive renorm}
\sum_{e \in \CE^\uparrow( \bar \CU)}(a_e + r_e - 1) - \sum_{e \in \CE^\downarrow(\bar \CU)} r_e 
=
\sum_{i \in J} \Big( 
\sum_{e \in \CE^\uparrow( \bar H^{(i)})} (a_e + r_e - 1) 
- 
\sum_{e \in \CE^\downarrow(\bar H^{(i)})} r_e \Big).
\end{equation}
%We also have
%\[
%\sum_{e \in \CE_{0}(\bar \CV)} a_{e} \le
%\sum_{j=1}^{n} \Big( \sum_{e \in \CE_0(\bar H^{(j)})} a_e \Big)
%+ \frac{|\mathfrak{s}|}{2} \sum_{j=1}^{n} |\bar H^{(j)}_{ex}|.
%\]
%since the first term on the RHS is precisely $\sum_{e \in \CE_{0}(\bar \CV) \setminus \CE_{c}(\bar \CV)} a_{e}$ while the second term on the RHS is an upper bound on $\sum_{e \in \CE_{c}(\bar \CV)} a_{e}$.

Applying item 3 of Assumption \ref{ass:ele-graph} and the bound \eqref{cumulant bound} with the summation set replaced by $J$ gives
\begin{equation*}
\begin{split}
&\sum_{e \in \CE_{0}(\bar \CU)} a_{e} 
+
\sum_{e \in \CE^\uparrow( \bar \CU)}(a_e + r_e - 1) - \sum_{e \in \CE^\downarrow(\bar \CU)} r_e\\
\le&
\sum_{i \in J}
\Big(  \!\!\!
\sum_{e \in \CE_{0}(\bar H^{(i)})}  \!\!\!   a_{e} 
+
  \!\!\!    \sum_{e \in \CE^\uparrow( \bar H^{(i)})}  \!\!\!  (a_e + r_e - 1) 
- \!\!\!\!\!   \sum_{e \in \CE^\downarrow(\bar H^{(i)})} \!\!\!   r_e 
+
\frac{|\mathfrak{s}|}{2}
|\bar H^{(i)}_{ex}|
\Big) \\
< & |\mathfrak{s}| \times
\sum_{i \in J}
\Big(
|\bar H^{(i)}_{in}|
+
|\bar H^{(i)}_{ex}|
\Big)
= |\mathfrak{s}| \, (|\bar \CU| - 1),
\end{split}
\end{equation*}
where for the last equality we note that $0$ is neither internal nor external. 

We now show \eqref{e:aEdges2} implies \eqref{e:assEdges2}. Suppose $\bar \CU \subset \CU \setminus \CU_{\star}$. Using the bound 
\[
\sum_{e \in \CE( \bar \CU) \cap \CE_{c}(\CU)}a_{e} \ge \sum_{i=1}^{p} |\bar{H}^{(i)}_{ex}|,
\] 
a decomposition similar to \eqref{sum positive renorm}, and applying \eqref{e:aEdges2} to each non-empty $\bar{H}^{(i)}$ gives
%\begin{equation*}
%\begin{split}
\begin{equs}
&\sum_{e \in \CE(\bar \CU)  }
 a_{e}
+
\sum_{e \in \CE^\uparrow(\bar \CU)} r_e
- \sum_{e \in \CE^\downarrow(\bar \CU)} (r_e-1)\\
& \ge
\sum_{i=1}^{p}
\bigg[
\sum_{e \in \CE(\bar{H}^{(i)}) \setminus \CE^\downarrow(\bar{H}^{(i)})} \!\!\!\!\!\! a_e
+  \!\!\!
\sum_{e \in \CE^\uparrow(\bar{H}^{(i)})} \!\!\! r_e
- \!\!\! \sum_{e \in \CE^\downarrow(\bar{H}^{(i)})} (r_e-1)
+
\frac{|\mathfrak{s}|}{2}
|\bar{H}^{(i)}_{ex}|
\bigg]\\
& >
|\mathfrak{s}| \times
\sum_{i=1}^{p} \Big( |\bar{H}^{(i)}_{in}| +  |\bar{H}^{(i)}_{ex}|
\Big)\ 
=\ |\mathfrak{s}| \times |\bar \CU|.
%\end{split}
%\end{equation*}
\end{equs}
\end{proof}
We now state a lemma which is a partial converse to the above theorem in the case of symmetric pairings of elementary graphs. A symmetric pairing of $H$ is a special type of Wick contraction - one has $p = 2$ and $\pi = \{\{v^{(1)}, v^{(2)}\}:\ v \in H_{ex}\}$ where for $v \in H$ we write $v^{(i)}$ for $v$'s instantiation in $H^{(i)}$. 
\begin{lemma} \label{lem:GG}
For a given elementary graph $H$, let $(\CV,\CE')$ be the reduced two-fold Wick contraction corresponding to the symmetric pairing of $H$. 
Suppose that $(\CV,\CE')$ satisfies Assumption~\ref{ass:big-graph}, then $H$ satisfies 
Assumption~\ref{ass:ele-graph}.
\end{lemma}

\begin{proof}
Clearly $H$ satisfies the first item of Assumption~\ref{ass:ele-graph} as a consequence of $(\CV,\CE')$ satisfying the first item of Assumption~\ref{ass:big-graph}.

For the other items we first observe that $\CV = H^{(1)}_{in} \sqcup H^{(2)}_{in} \sqcup \{0\}$ and 
\begin{equ} \label{e:GGobserve}
\CE' \setminus \left[ \CE_{0}(H^{(1)}_{in}) \sqcup \CE_{0}(H^{(2)}_{in}) \right]
=
\cup_{v \in H_{ex}} \{ i^{(1)}(v), i^{(2)}(v) \}
\end{equ}
where $i^{(j)}(v)$ is the instantiation of $i(v)$ in $H^{(j)}_{ex}$.

For item $2$ fix some $\bar{H} \subseteq H_{0}$.  
%item 2. of Assumption~\ref{ass:ele-graph}.
If $\bar H \cap H_{ex} = \emptyset$ then \eqref{e:aEdges}
follows by applying \eqref{e:assEdges} to $\bar\CV= \bar H^{(1)}$ -
the right hand side of \eqref{e:aEdges} is equal to
$|\s|\,(|\bar H|-1)$.

We claim that given appropriate $\bar{H}\subset H$, the needed criteria for the other case of item $2$ (where $\bar{H} \cap H_{ex} \not = \emptyset$) or items 3 or 4 are equivalent to the corresponding items of Assumption \ref{ass:big-graph} for the set $\bar{\CV}\subset \CV$, where  $\bar{\CV} \eqdef \bar{H}_{in}^{(1)} \sqcup \bar{H}_{in}^{(2)}$ if $\bar{H} \not \ni 0$ and $\bar{\CV} \eqdef \bar{H}_{in}^{(1)} \sqcup \bar{H}_{in}^{(2)} \sqcup \{0\}$ if $\bar{H} \ni 0$.

Since the arguments are so similar we only do the other case of item 2 here. Suppose $\bar{H} \subset H_{0}$ with $\bar H \cap H_{ex} \neq \emptyset$. 
Defining $\bar\CV$ as described and then applying \eqref{e:assEdges} 
together with \eqref{e:GGobserve}
 to $\bar \CV$ gives
\begin{equ}
2
\sum_{e\in\CE_0(\bar{H})} a_e 
- |\s| \times   
|\bar{H}_{ex}|
= 
\sum_{e \in \CE_0(\bar\CV)}
a_{e}
< |\s| ( |\bar \CV|-1) = |\s| ( 2 |\bar{H}_{in}| - 1 ),
\end{equ}
which is exactly the desired condition for $\bar H$.
\end{proof}

\begin{remark} The above lemma establishes a type of ``hypercontractive'' 
or ``equivalence of moments'' bound in the non-Gaussian setting.
\end{remark}

We state two lemmas and a remark before proceeding to the stochastic estimates.

\begin{lemma} \label{lem:merge}
Suppose that the elementary graph $H=(\CV,\CE)$ satisfies Assumption~\ref{ass:ele-graph},
and that $H'=(\CV,\CE')$ is another elementary graph 
which has the same vertex set $\CV$.
Assume that there exist internal edges $e_1,e_2\in \CE$ with $a_{e_j}=|e_j||\s|/2$ and $r_{e_j}\le 0$
for $j=1,2$, and that $\CE'=(\CE \setminus \{e_1,e_2\}) \cup e$
with $e=e_1\cup e_2$ and $a_e = |e||\s|/2$.
In plain words $H'$
is formed by merging $e_1,e_2$ into one hyper-edge $e$. 
Then, $H'$ also satisfies Assumption~\ref{ass:ele-graph}.
\end{lemma}

\begin{proof}
For items 1-3 of Assumption~\ref{ass:ele-graph} and any allowable subset $\bar H\subset \CV$, the LHS of the bounds for $\bar H$ as a subgraph of $H'$ is always smaller or equal to 
the LHS  for $\bar H$ as a subgraph of $H$, while the RHS remains the same.

For item 4 and any allowable subset $\bar H$, the LHS of the bound for $\bar H$ as a subgraph of $H'$ is always larger or equal to 
the LHS  for $\bar H$ as a subgraph of $H$, while the RHS again remains the same.

Therefore if $H$ satisfies Assumption~\ref{ass:ele-graph}, then after merging $e_1,e_2$ into one hyper-edge the new graph also satisfies Assumption~\ref{ass:ele-graph}.
\end{proof}

\begin{lemma} \label{lem:no-ex}
For graphs $H$ such  that $H_{ex}=\emptyset$,
Assumption~\ref{ass:ele-graph}  for $H$ is equivalent with Assumption~\ref{ass:big-graph} for $\CV=H$.
\end{lemma}

\begin{proof}
Immediate upon comparing the two assumptions with the condition $H_{ex}=\emptyset$.
\end{proof}

\begin{remark} %\ajay{Added remark about bad chains in elementary graphs}
When we do our stochastic estimates there are symbols $\tau$ and self-contractions $\pi$ on $\tau$ such that the elementary graph $H_{\tau,\pi}$ will have a bad chain, failing to abide Assumption \ref{ass:ele-graph}. 
However, it is easy to see that the finite set of bad chains of $H_{\tau,\pi}$ can be eliminated by integrating a noise vertex in each one yielding an abiding $\tilde{H}_{\tau,\pi}$. 
Below this is done implicitly whenever the scenario arises.
\end{remark}

\section{Application: Wong-Zakai theorem for non-Gaussian noise}
\label{sec:AppWZ}

We now apply the machinery 
of the preceding sections
to prove Theorem~\ref{theo:main}.
%Let $u$ be the It\^o solution to
% the following stochastic PDE driven by space-time white noise
%\begin{equ}[e:SPDE]
%du = \d_x^2 u\,dt + H(u)\,dt + G(u)\,dW(t)\;.
%\end{equ}

\subsection{Renormalization}

%\hao{What's the best way to introduce $\CW_0$ and its span $\CT_0$? (Or rather avoid defining them?)}
%\ajay{Hmm, does this work? I think we should just precisely cite what we are using, then skipping some technical details is okay}
We fix $\CT_{0} \subset \CT$ to be the span of $\CW_{0}$ where
\[
\CW_0 \eqdef \Big\{\<Xi>, \<Xi2>, \<Xi3>, \<Xi3b>, \<XiX>, \<Xi4>, \<Xi4c>, \<Xi4e>,\<Xi4b>,\<Xi2X>,\<XXi2>,\1,\<IXi^2>,\<IXi2>,\<Xi22>, \<IXi>\Big\} \;.
\]
Our maps $M: \CT_{0} \mapsto \CT_{0}$ will be of the form 
\begin{equation}\label{form of M}
M = \exp(- \sum_{i=1}^7 \ell_i L^{(i)} )
\end{equation} 
where $\{\ell_i\}_{i=1}^{7} \subset \R$ and the $\{L^{(i)}\}_{i=1}^{7}$ 
are nilpotent linear operators on $\CT_0$ given as follows. $L^{(1)}$ is defined in the same way as the map called in $L$ in \cite{WongZakai}: it iterates over all occurrences of $\<Xi2>$ as a ``subsymbol'' of $\tau$ 
and ``erases'' it in the graphical notation:
\begin{equs}[3][e:defL]
L^{(1)}\<Xi2>&=\1\;, \quad&\quad
L^{(1)}\<Xi3>&=\<IXi>\;, \quad&\quad
L^{(1)}\<Xi3b>&=2\,\<IXi>\;, \\
L^{(1)}\<Xi4b>&=3\,\<IXi^2>\;, \quad&\quad
L^{(1)} \<Xi4e> &= \<IXi2> + \<IXi^2>\;, \quad&\quad
L^{(1)}\,\<Xi4>&=\<IXi2> + \<Xi22>\;,\\
L^{(1)} \<Xi4c> &= \<IXi^2> + 2\, \<Xi22>\;, \quad&\quad
L^{(1)}\<Xi2X>&= \X_1\;,  \quad&\quad L^{(1)}\<XXi2>&= \X_1\;.
\end{equs}
Note that  there is no term $\<XiI>$ appearing in $L^{(1)}\<Xi3>$ because $\CI(\1) = 0$, and
similarly for the other terms.
We furthermore have  $L^{(1)}\tau =0$ for every $\tau\in\CW_0$ which is not one of the above nine elements.
%\begin{equ}
%L^{(1)}\1 = L^{(1)}\,\<Xi> = L^{(1)}\, \<XiX> = L^{(1)}\<Xi22> = L^{(1)}\<IXi^2> = L^{(1)} \<IXi2> = L^{(1)}\<IXi> = 0\;.
%\end{equ}
%\hao{About $M\in \mathfrak R$, $L^{(1)}$ is the same as HP so no proof needed. I think for other $L^{(i)}$, one always has $\Delta^M \tau = M\tau \otimes 1$. This should immediately imply upper-triangular, no?}
%\ajay{Yes, this is right.}
Regarding the other maps, one has
\[
L^{(2)} \<Xi3> = \1\;, \quad  L^{(2)} \<Xi4> = \<IXi>\;, \quad
	L^{(2)}\<Xi4e>=\<IXi>
\]
and $L^{(2)} \tau = 0 $ for 
	every $\tau\in \CW_0 \setminus \{\<Xi3>,\<Xi4>,\<Xi4e>\}$, and
\[
L^{(3)} \<Xi3b> = \1\;, \quad L^{(3)} \<Xi4b> = 3\<IXi>\;, \quad
	L^{(3)}\<Xi4e>=\<IXi>
\]
and $L^{(3)} \tau = 0$ for every $\tau\in \CW_0 \setminus \{\<Xi3b>, \<Xi4b>,\<Xi4e>\} $, and
\[
L^{(4)} \<Xi4> = \1 \;, \quad
L^{(5)} \<Xi4e> = \1\;, \quad
L^{(6)} \<Xi4b> = \1\;, \quad
L^{(7)} \<Xi4c> = \1\;.
\]
and $L^{(i)}\tau =0$  for every $\tau\in \CW_0 \setminus \{\tau\}$ 
with $(i,\tau)\in \{(4,\<Xi4> ),(5,\<Xi4e> ),(6,\<Xi4b> ),(7,\<Xi4c>) \}$.
%  Below is the old way
%\begin{equs}
%L^{(1)} \<Xi4> &= \1 \;, \quad \mbox{and } 
%	L^{(1)} \tau = 0 \mbox{ for } \forall\tau \in \CW_0 \setminus \{\<Xi4>\} \\
%L^{(2)} \<Xi4e> &= \1\;, \quad \mbox{and } 
%	L^{(2)} \tau = 0 \mbox{ for } \forall\tau \in \CW_0 \setminus \{\<Xi4e>\} \\
%%
%L^{(5)} \<Xi4b> &= \1\;, \quad \mbox{and } 
%	L^{(5)} \tau = 0 \mbox{ for } \forall\tau \in \CW_0 \setminus \{\<Xi4b>\} \\
%L^{(6)} \<Xi4c> &= \1\;, \quad \mbox{and } 
%	L^{(6)} \tau = 0 \mbox{ for } \forall\tau \in \CW_0 \setminus \{\<Xi4c>\} \\
%\end{equs}

Recall that $\mathfrak R$ is the renormalization group introduced in Theorem~\ref{thm: renormalization operator}.
\begin{theorem}
For any choice of the constants $\ell_{j}$, and $M \in \mathfrak{R}$.
\end{theorem}
\begin{proof}
One can check that $L^{(i)}L^{(j)} \tau = 0$ for
for all $\tau \in \CW_0$ and any $i,j$ - thus the operators $L^{(i)}$ all commute and one actually has $M = I - \sum_{i=1}^7 \ell_i L^{(i)}$. Furthermore since $\mathfrak{R}$ is a group it suffices to check, for $1 \le j \le 7$, the upper-triangularity of $\Delta^{M_{j}}$ where $M_{j} \eqdef e^{-\ell_{j}L_{j}}$. This can be checked by computation.

The $j=1$ case requires the most work but the computations for this case are exactly the same as those found in \cite{WongZakai}[Sec 4.2] for the operator called $L$.

For $j \ge 2$ we observe that one has, for all $\tau \in \CT_{0}$, $\hat{M}_{j} \CJ_{k}(\tau) = \CJ_{k}(M_{j} \tau)$ and $\Delta^{M_{j}} = M_{j} \otimes \mathrm{Id}$. Clearly $\Delta^{M_{j}}$ is upper triangular. 
\end{proof}

%For a quantity $Q_\eps$ we write $Q_\eps \approx \eps^{-\alpha}$
%as a shorthand for the statement 
%``there exist constants $Q$ such that $\lim_{\eps \to 0} |Q_\eps - Q\eps^{-\alpha}|=0$". ({\bf probably not useful})

%{\bf Prove or say something on $M\in \mathfrak R$.}

We then define, for each $\eps \in (0,1]$, a map $M_\eps \in \mathfrak{R}$ by specifying the constants
\begin{equs}[e:choiceell]
\ell_1 &= \eps^{-1} C^{\<Xi2>} \;,\quad
\ell_2 =\eps^{-\frac12} C^{\<Xi3>}     \;,\quad
\ell_3 = \eps^{-\frac12} C^{\<Xi3b>}  \;,\\
\ell_4 & = c^{(2)}_{\zeta}  \;,\quad
\ell_5 = c^{(4)}_{\zeta}\;, \quad
\ell_6 = c^{(1)}_{\zeta}\;,\quad
\ell_7 = c^{(3)}_{\zeta} \;.
\end{equs}
%\begin{equs} 
%C^{(1)}_{\zeta} &= C^{\<Xi2>}\qquad
%C^{(2)}_{\zeta} = C^{\<Xi3>}\qquad
%C^{(3)}_{\zeta} = C^{\<Xi3b>} \\
%c^{(1)}_{\zeta} &= C^{\<Xi4b>} \qquad
%c^{(2)}_{\zeta} = C_{1}^{\<Xi4>} +C_{2}^{\<Xi4>} +C_{3}^{\<Xi4>} \\
%c^{(3)}_{\zeta} &= C^{\<Xi4c>}  \qquad
%c^{(4)}_{\zeta} = C_{1}^{\<Xi4e>}+C_{2}^{\<Xi4e>}+C_{3}^{\<Xi4e>}
%\end{equs}
Recall that the constants on the right hand side are defined in \eqref{e:def-consts}.

\subsection{Moment bounds}
In order to identify the limit of our sequence of models $(\hat \Pi^{(\eps)}, \hat \Gamma^{(\eps)})$ as the It\^o model 
we follow the approach in \cite{CLTKPZ} and introduce another level of regularization $\zeta_{\eps,\bar\eps} = \rho_{\bar\eps}*\zeta_\eps$ with $\bar\eps>0$ and $\rho_{\bar\eps} = \bar\eps^{-3} \rho (\bar\eps^{-2} t, \bar\eps^{-1} x)$, where $\rho$
 is smooth, compactly supported, even in the space variable and integrating to one.
 
 We then construct the  renormalized model $ (\hat \Pi^{(\eps,\bar\eps)}, \hat \Gamma^{(\eps,\bar\eps)})$  from $\zeta_{\eps,\bar\eps}$
together with the renormalization maps $M_{\eps,\bar\eps}$
with constants specified in the same way as in \eqref{e:choiceell},
except that we replace every $P$ by our truncation $K$, every $\fC_n$ by $\fC^{(\eps)}_n$, and finally drop the factors $\eps^{-1}$ and 
$\eps^{-\frac12}$ in \eqref{e:choiceell}.
For example, the constant 
$\eps^{-1} C^{\<Xi2>} $ in  \eqref{e:choiceell}  is replaced by 
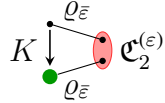
\begin{equ}
\begin{tikzpicture}[scale=0.35,baseline=0.5cm]
	\node at (-2,1)  [root] (left) {};
	\node at (-2,3)  [int] (left1) {};
	
	\draw[kernel] (left1) to (left); \node at (-3,2) {$K$};
	
	\node[cumu2n]	(a)  at (0,2) {};	
	\draw[cumu2] (a) ellipse (10pt and 20pt);
		\node at (1.5, 2) {$\fC_2^{(\eps)}$}; 
	
	\draw ($(a)+(-90:4mm)$) node[int]   {} to  (left);
		\node at (-1,3.5) {$\rho_{\bar\eps}$};  
	\draw ($(a)+(90:4mm)$) node[int]   {} to (left1);
		\node at (-1,0.5) {$\rho_{\bar\eps}$};  
\end{tikzpicture}
\end{equ}
%Here each line represents a kernel, 
%with 
%\tikz[baseline=-0.1cm] \draw[kernel] (0,0) to (1,0);
%representing the kernel $K$ (rather than the kernel $P$ represented by 
%\tikz[baseline=-4] \draw[Pkernel] (0,0) to (1,0);
%in the preceding sections), and
%\tikz[baseline=-0.1cm] \draw[rho] (0,0) to (1,0);
%epresenting the kernel $\rho_{\bar\eps}$.
%A polygon with $n$ points inside
% \begin{tikzpicture}[scale=0.4] 
% \node[cumu4] at (0,0) {}; 
% \node at (-0.35,0.35) [dot] {}; 
% \node at (-0.35,-0.35) [dot] {}; 
% \node at (0.35,0.35) [dot] {}; 
% \node at (0.35,-0.35) [dot] {}; 
% \end{tikzpicture}
% ($n=4$ here) represents the cumulant $\fC_{n}^{(\eps)}(z_{1}, \cdots, z_{n})$.
 
One may find that for a fixed $\eps>0$, sending  $\bar\eps$  to zero does not exactly recover the renormalization constants for $(\hat \Pi^{(\eps)}, \hat \Gamma^{(\eps)})$; for instance in the latter model the renormalization constants are defined via $P$.
This does not matter for two reasons: (i) we will only consider the situation that $\eps$ is much less than $\bar\eps$;
(ii) when we bound the moments for 
$(\hat \Pi^{(\eps)}_x \tau) (\varphi_x^\lambda)$ below,
we will actually replace the constants for $\hat\Pi^{(\eps)}$
by the ones defined via $K$ and $\fC_{n}^{(\eps)}$ here, with an error
that goes to zero as $\eps\to 0$. More precisely,
for the constants  $\eps^{-1}C^{\<Xi2>},\eps^{-\frac12} C^{\<Xi3>},\eps^{-\frac12}C^{\<Xi3b>} $, by exponential decay of $\fC_n$ and the fact that $P(z)=K(z)$ for $|z|<1$, one can easily see by a scaling argument that this error is bounded by $\theta^{\frac{1}{\eps}}$ with $\theta\in (0,1)$;
thus even though these errors are multiplied by ``graphs" which may diverge as $\eps\to 0$ they still vanish in that limit.
For the other constants one can also argue as in \cite{WongZakai}
that the error of such replacement vanishes as $\eps\to 0$.
 
% {\bf In all the following graphs, constants need to be changed to the blue ones. Need to say dotted lines are $\delta$.}

\begin{proposition} 
\label{prop:mombound}
Let $(\hat \Pi^{(\eps)}, \hat \Gamma^{(\eps)})$ and $ (\hat \Pi^{(\eps,\bar\eps)}, \hat \Gamma^{(\eps,\bar\eps)})$ be defined as above.
 % with renormalisation constants given by \eqref{e:choiceell}.
There exist $\kappa,\eta>0$   such that 
for every $\tau\in\{\<Xi>,\<Xi2>,\<Xi2X>,\<XXi2>,\<Xi3>,\<Xi3b>,\<Xi4>,\<Xi4b>,\<Xi4c>,\<Xi4e>\}$
 and every   $p>0$,
one has
\begin{equ}  \label{e:mombound}
\E | (\hat\Pi^{(\eps)}_x \tau  -  \hat\Pi^{(\eps,\bar\eps)}_x \tau ) (\varphi_x^\lambda) |^p  
	\lesssim
	\bar\eps^\kappa
	 \lambda^{p (|\tau|_\s + \eta)} 
\end{equ}
uniformly in all $\bar\eps\in(0,1]$ and all $\eps\in(0,1]$ sufficiently small depending on $\bar\eps$, all $\lambda\in(0,1]$, all test functions $\varphi\in\CB$, and all $x\in\R^2$.
\end{proposition}

\begin{remark}
We will actually mainly focus on a weaker bound 
\begin{equ} \label{e:momWeak}
\E | (\hat \Pi^{(\eps)}_x \tau) (\varphi_x^\lambda) |^p  
	\lesssim \lambda^{p (|\tau| + \eta)} \;,
\end{equ}
and \eqref{e:mombound} will turn out to follow in the same way.
\end{remark}

%\hao{Do you agree: we'll focus on bounding $\hat \Pi^{(\eps)}$,
%so in all the expressions below we should write 
%constants like  $\eps^{-\frac12} C^{\<Xi3>}$,
%not like $C^{\<Xi3>}_\eps$. Also our dotted lines are deltas. Just want to confirm with you.}
%\ajay{Sounds good to me.}
%\hao{Changed.}

We introduce more graphical notation.
We denote by
\tikz[baseline=-0.1cm] \draw[testfcn] (1,0) to (0,0);
 a generic test function $\phi_\lambda$ rescaled to 
scale $\lambda$
and label it by $(0,0)$.
An arrow
\tikz[baseline=-0.1cm] \draw[kernel] (0,0) to (1,0);
labeled by $(1,0)$
represents the kernel $K$.
A barred arrow 
\tikz[baseline=-0.1cm] \draw[kernel1] (0,0) to (1,0);
labeled by $(1,1)$
represents a function $K(t-s,y-x) - K(-s,-x)$, where
$(s,x)$ and $(t,y)$ are the coordinates of the starting and end
point respectively. 
A double barred arrow 
\tikz[baseline=-0.1cm] \draw[kernel2] (0,0) to (1,0);
labeled by $(1,2)$
represents a function
$K(t-s,y-x) - K(-s,-x) - y\, K'(-s,-x)$. 
The dotted lines
\tikz[baseline=-0.1cm] \draw[rho] (0,0) to (1,0);
represent the Dirac distributions and are labeled by $(3,-1)$.
A polygon with $n$ points inside
 \begin{tikzpicture}[scale=0.4] 
 \node[cumu4] at (0,0) {}; 
 \node at (-0.35,0.35) [dot] {}; 
 \node at (-0.35,-0.35) [dot] {}; 
 \node at (0.35,0.35) [dot] {}; 
 \node at (0.35,-0.35) [dot] {}; 
 \end{tikzpicture}
 ($n=4$ here) represents the cumulant $\fC_{n}^{(\eps)}(z_{1}, \cdots, z_{n})$
 and is labeled by $(-3n/2,0)$.

We also draw an arrow  \tikz[baseline=-0.1cm] \draw[kprime] (0,0) -- (1,0);
with label $(2,0)$ for the kernel $K' = \d_x K$ 
and  \tikz[baseline=-0.1cm] \draw[testfcnx] (1,0) -- (0,0);
for the test function $(t,x) \mapsto x\phi^\lambda(t,x)$.
Note that $\tilde \phi(t,x) = x \phi(t,x)$ is again an 
admissible test function and one has $x\phi^\lambda(t,x) = \lambda \tilde \phi^\lambda(t,x)$.
As a consequence, when a test function $\phi$ is replaced with test function $\tilde \phi$,
one gains an additional power of $\lambda$. 

%Below we prove the bounds for the objects listed in Prop.~\ref{prop:mombound} one by one. We will frequently use 

\subsubsection[Term Xi2]{The symbols \texorpdfstring{$\<Xi2>$}{Xi2}\,, \texorpdfstring{$\<Xi2X>$}{Xi2X}\,,
and \texorpdfstring{$\<XXi2>$}{XXi2}}

We start with the simplest object $\<Xi2>$ after the noise. By translation invariance we take the point $x$  in \eqref{e:momWeak} to be $0$. 
Using Definition~\ref{def:Wick-product} which allows us to represent a product of the noises as Wick products, one has
\begin{equ} [e:decPiXiTwo]
\bigl(\hat \Pi_0^{(\eps)} \<Xi2>\bigr)(\phi_0^\lambda) = \;
\begin{tikzpicture}[scale=0.35,baseline=0.3cm]
	\node at (0,-1)  [root] (root) {};
	\node at (-2,1)  [dot] (left) {};
	\node at (-2,3)  [dot] (left1) {};
	\node at (0,1) [var] (variable1) {};
	\node at (0,3) [var] (variable2) {};
	
	\draw[testfcn] (left) to  (root);
	
	\draw[kernel1] (left1) to (left);
	\draw[rho] (variable2) to (left1); 
	\draw[rho] (variable1) to (left); 
\end{tikzpicture}\;
+ \;
\begin{tikzpicture}[scale=0.35,baseline=0.3cm]
	\node at (0,-1)  [root] (root) {};
	\node at (-2,1)  [dot] (left) {};
	\node at (-2,3)  [dot] (left1) {};
	
	\node[cumu2n]	(a)  at (0,2) {};	
	\draw[cumu2] (a) ellipse (10pt and 20pt);
	
	\draw[testfcn] (left) to (root);
	
	\draw[kernel1] (left1) to (left);
	\draw[rho] ($(a)+(90:4mm)$) node[int]  {} to (left1); 
	\draw[rho] ($(a)+(-90:4mm)$) node[int]  {} to (left); 
\end{tikzpicture}\;
-\;\eps^{-1} C^{\<Xi2>} \;\;
\begin{tikzpicture}[scale=0.35,baseline=0.5cm]
	\node at (0,1)  [root] (root) {};
	\node at (0,3.2) [int] (left) {};
	\draw[testfcn] (left) to  (root);
\end{tikzpicture}
\;\;=\;
\begin{tikzpicture}[scale=0.35,baseline=0.3cm]
	\node at (0,-1)  [root] (root) {};
	\node at (-2,1)  [dot] (left) {};
	\node at (-2,3)  [dot] (left1) {};
	\node at (0,1) [var] (variable1) {};
	\node at (0,3) [var] (variable2) {};
	
	\draw[testfcn] (left) to  (root);
	
	\draw[kernel1] (left1) to (left);
	\draw[rho] (variable2) to (left1); 
	\draw[rho] (variable1) to (left); 
\end{tikzpicture}\;
 - \;
\begin{tikzpicture}[scale=0.35,baseline=0.3cm]
	\node at (0,-1)  [root] (root) {};
	\node at (-1,1)  [dot] (left) {};
	\node at (1,1) [dot] (right) {};
	
	\draw[testfcn] (left) to  (root);
	
	\node[cumu2n]	(a)  at (0,3) {};	
	\draw[cumu2] (a) ellipse (20pt and 10pt);
	
	\draw[kernel] (right) to (root);
	\draw[rho] ($(a)+(0:4mm)$) node[int]  {} to (right); 
	\draw[rho] ($(a)+(180:4mm)$) node[int]  {} to (left); 
\end{tikzpicture}
+E_\eps
\end{equ}
where $E_\eps$ here and below is an error which arises from the difference between 
$\eps^{-1} C^{\<Xi2>} $ and $ C_\eps^{\<Xi2>} $ and
goes to zero.
%\hao{Check if we have said enough about why $E_\eps\to 0$.}
%\ajay{We should also say that it is $O(\lambda^{|\tau|})$ in $\lambda$ right? In this case it is actually $O(1)$ in $\lambda$}

Now if we compare the above graphs with the corresponding ones in \cite[Eq.~(5.2)]{WongZakai} in the Gaussian noise case,
we realize that they are essentially the same graphs.
Since in \cite{WongZakai} it is checked that the symmetric pairing of the first graph, as well as the last graph above,
satisfy Assumption~\ref{ass:big-graph}, our Lemma~\ref{lem:GG}
and Lemma~\ref{lem:no-ex}  immediately yield Assumption~\ref{ass:ele-graph}
for the two graphs on the RHS of \eqref{e:decPiXiTwo},
and therefore by Theorem~\ref{theo:ele-graph} one concludes
 the desired bounds for $\bigl(\hat \Pi_0^{(\eps)} \<Xi2>\bigr)(\phi_0^\lambda) $.

The other two symbols $\<Xi2X>$
and $\<XXi2>$ are bounded in the same way.
In these cases, graphs appearing in the expansions are the same
with those in the Gaussian case. 
However, in general, for other symbols below,
there will be new terms due to the nontrivial higher cumulants of our non-Gaussian noise (we sometimes refer to them as just ``new terms").
The discussion here shows that we only need to treat with these new terms.

\subsubsection[Term Xi3]{The symbol \texorpdfstring{$\<Xi3>$}{Xi3}}

%{\bf 2 graphs to check}
Besides the terms appearing in \cite{WongZakai},
we have the following new terms in the expression for 
$\bigl(\hat \Pi_0^{(\eps)} \<Xi3>\bigr)(\phi_\lambda)$:

\begin{equ}

\end{equ}
It is straightforward to check that Assumption~\ref{ass:ele-graph} is satisfied for the two graphs on the RHS, which yields the desired bounds
$\E | (\hat \Pi^{(\eps)}_x \<Xi3>) (\varphi_x^\lambda) |^p  
	\lesssim \lambda^{p (-\frac12+ \eta)}$.

\subsubsection[Term Xi3b]{The symbol \texorpdfstring{$\<Xi3b>$}{Xib3}}

%{\bf 2 graphs to check}
We now turn to $\<Xi3b>$. In this case, besides the terms shown in \cite{WongZakai},
we have the following new terms in the expression for 
$\bigl(\hat \Pi_0^{(\eps)} \<Xi3b>\bigr)(\phi_\lambda)$:
\begin{equ}
%
\;+\;
E_\eps
\end{equ}
One can also
check that Assumption~\ref{ass:ele-graph} holds for these two graphs.

\subsubsection[Term Xi4]{The symbol \texorpdfstring{$\<Xi4>$}{Xi4}}

%{\bf 8 graphs to check}
The new terms for $\bigl(\hat \Pi_0^{(\eps)} \<Xi4>\bigr)(\phi_\lambda)$ are:
\begin{equs}  
%
\;+\;E_\eps
\end{equs}
The sum of the last two terms can be written as 
a sum of 11 graphs (expanding the terms represented by the barred arrows yield 12 terms, and the renormalization constant cancels one of them);
each of them has a fourth order hyper-edge
which can be split into two edges
as
 %
.
After this split all the graphs satisfy Assumption~\ref{ass:big-graph}
as checked in \cite{WongZakai} so they satisfy Assumption~\ref{ass:ele-graph} by Lemmas~\ref{lem:no-ex} and \ref{lem:merge}.

After cancellations with renormalization constant
we only need to check the assumption for the following graphs
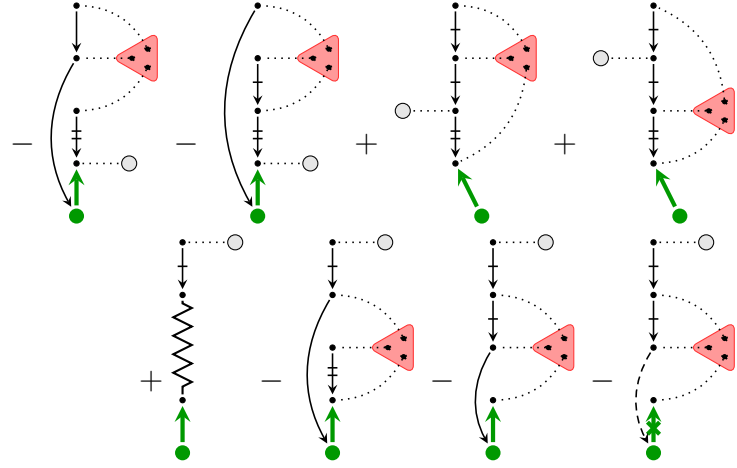
\begin{equs}  
\;-\;
%
\end{equs}
where \tikz \draw[kernelBigg] (0,0) to (1.6,0); represents 
the kernel
\begin{equ}
%
\; - \; C_\eps^{\<Xi3>}\; \delta(z)
\end{equ}
understood as a function of $z$,
with label $(3\frac12,-1)$. It can be checked that these eight graphs all satisfy the conditions in Assumption~\ref{ass:ele-graph}.

%\hao{There are several zigzag lines in this section: need different graphic notations.}

\subsubsection[Term Xi4e]{The symbol \texorpdfstring{$\<Xi4e>$}{Xi4e}}

%{\bf 7 graphs to check}
The new terms are
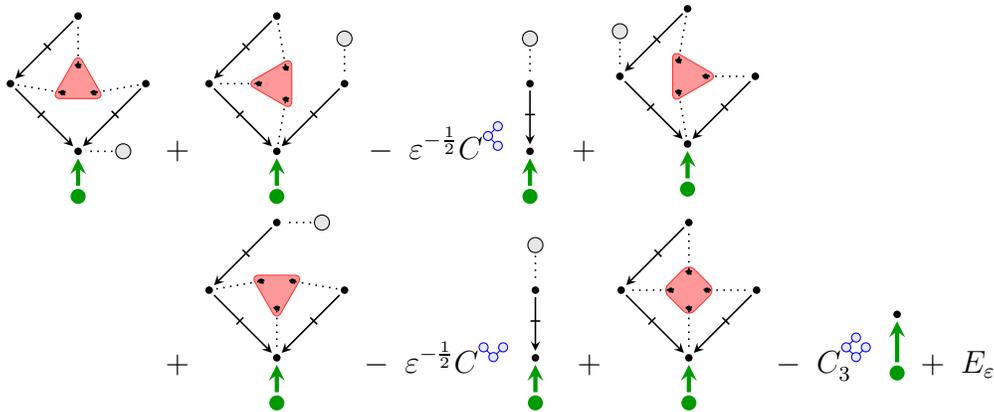
\begin{equs}  
%\bigl(\hat \Pi_0^{(\eps)} \<Xi4e>\bigr)(\phi_\lambda) &= \; 
%
\;+\;E_\eps
\end{equs}
For the last two terms, after cancellation with $C_{3,\eps}^{\<Xi4e>} $, there are
seven terms, each having a hyper-edge that can be split
into two edges representing a cumulant of the top and bottom
noises and a cumulant of the left and right noises.
These terms all satisfy Assumption~\ref{ass:big-graph}
as checked in \cite{WongZakai}. 

So we are left with 
\begin{equs}  
%\bigl(\hat \Pi_0^{(\eps)} \<Xi4e>\bigr)(\phi_\lambda) &= \; 
%
\end{equs}
where \tikz \draw[kernelBigg1] (0,0) to (1.6,0); represents the kernel
%{\bf something different than above.. should use another notation}
\begin{equ}
%
\; - \; \eps^{-\frac12}C^{\<Xi3b>}\; \delta(z)
\end{equ}
labeled by $(3\frac12,-1)$. It can be checked that these seven graphs all satisfy the conditions in Assumption~\ref{ass:ele-graph}.

\subsubsection[Term Xi4b]{The symbol \texorpdfstring{$\<Xi4b>$}{Xi4b}}

%{\bf 6 graphs to check}
The new terms are
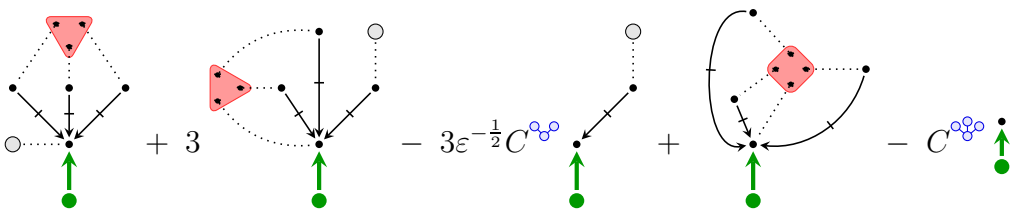
\begin{equ}  
%\bigl(\hat \Pi_0^{(\eps)} \<Xi4b>\bigr)(\phi_\lambda) = \; 
%
\end{equ}
The sum of the 2nd and the 3rd terms is %\ajay{Modulo an overall factor of 3 }
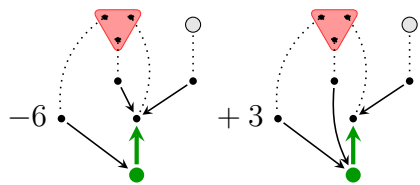
\begin{equ}
\;-6\;
%
\end{equ}
The sum of the last two terms is
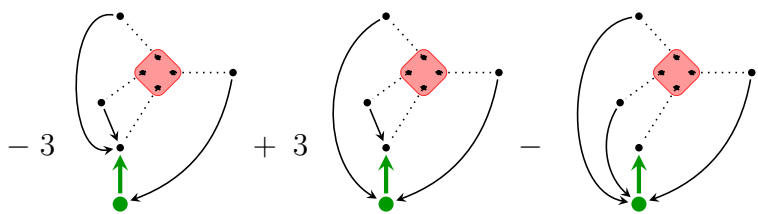
\begin{equ}
-\;3\;
%
\end{equ}
Therefore, for this symbol, six graphs remain to be checked. Indeed, they all satisfy Assumption~\ref{ass:ele-graph}.

\subsubsection[Term Xi4c]{The symbol \texorpdfstring{$\<Xi4c>$}{Xi4c}}

%{\bf 9 graphs to check}
The new terms are %\ajay{First two terms should be multiplied by 2}
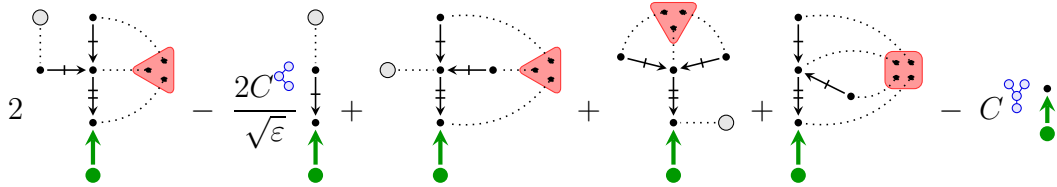
\begin{equs}  
%\bigl(\hat \Pi_0^{(\eps)} \<Xi4c>\bigr)(\phi_\lambda) &= \; 
\;2\;
%
\end{equs}
The sum of the first two terms is $2$ times
%\ajay{In the last two terms, the uppermost vertical kernel line should \emph{not} have a bar right?}
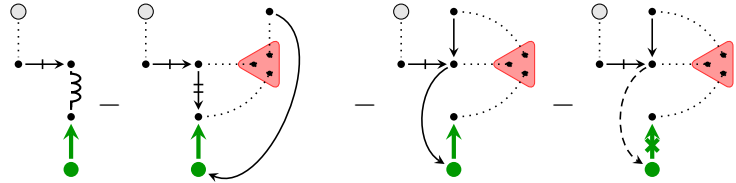
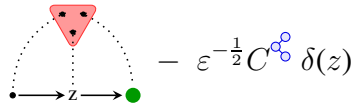
\begin{equs}  
%
\; - \; \eps^{-\frac12}C^{\<Xi3>}\; \delta(z)
\end{equ}
labeled by $(3\frac12,-1)$. 

The 4th term is equal to (noting that the heat kernels in \tikz[baseline=-0.1cm] \draw[kernel2] (0,0) to (1,0); annihilate constants)
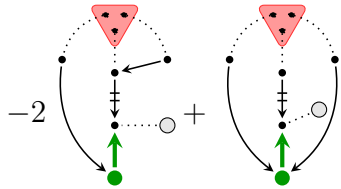
\begin{equ}
-2\;
%
\end{equ}
Finally the sum of the last two terms is

%{\bf You are right, something was wrong above. Below should be correct.}

%Using the identity
%\[
%(K-K_1)^2 (H-H_1-H_2) -K^2 H
%= -(K-K_1)^2 H_1 -(K-K_1)^2 H_2 
%+ K_1^2 H -2KK_1 H
%\]
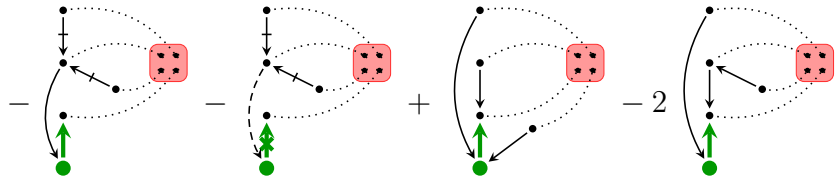
\begin{equ}
\;-\;
%
\end{equ}

The first two graphs above have a fourth order hyper-edge, by splitting it into two edges and using Lemma~\ref{lem:merge} they become the graphs checked in \cite{WongZakai}. 

Therefore, for this symbol, there are still nine graphs  to be checked. It is straightforward (though a bit tedious) to check that they all satisfy Assumption~\ref{ass:ele-graph}.

\begin{proof}[of Proposition~\ref{prop:mombound}]
Collecting all the results of this section, we obtain the weaker bound  
 \eqref{e:momWeak}.
The bound \eqref{e:mombound},
follows in essentially the same way as 
the first bound, by the argument in \cite{CLTKPZ} (see also
the verification of the second bound 
in \cite[Theorem~10.7]{Regularity}).
 Indeed, as we consider the difference between 
$\hat\Pi_{0}^{(\eps)} \tau$
and $ \hat\Pi_{0}^{(\eps,\bar\eps)} \tau$
for any $\tau \neq \Xi$, we obtain the same graphs as above,
and  in each graph some of the instances of $\delta$
are replaced by $\rho_{\bar\eps}$
and exactly one instance is replaced by $\delta-\rho_{\bar\eps}$.
By the bound
$\| \delta-\rho_{\bar\eps}\|_{-3-\kappa} \lesssim \bar\eps^\kappa$
one obtains the same bound as  \eqref{e:momWeak}
with an extra factor $\bar\eps^\kappa$, which is exactly required by the bound \eqref{e:mombound}.
The case $\tau = \Xi$ can be also proved as in \cite{CLTKPZ}.
\end{proof}

\subsection{Proof of the main theorem}
%
%{\bf Need to say something on CLT for the noise... careful about exponential dacay...}
%
It was shown in \cite{WongZakai}  that if we replace
$\zeta_\eps$ by $\xi_\eps$ with
$\xi_\eps=\rho_\eps * \xi$ where $\xi$
is the Gaussian space-time white noise,
the renormalised models built from $\xi_\eps$ converge to 
the  limit $\hat Z=(\hat \Pi, \hat\Gamma)$ called the It\^o model, 
which satisfies the following property.
For every $\tau \in \CU$ and every $(t,x)$, the process 
$s \mapsto \bigl(\hat \Pi_{(t,x)}\tau\bigr)(s,\cdot)$ is $\CF_s$-adapted for $s > t$
and, for every smooth test function $\phi$ supported in the future $\{(s,y)\,:\, s > t\}$,
 one has the identity
\begin{equ}[e:Itointegral]
\bigl(\hat \Pi_{(t,x)}\Xi\tau\bigr)(\phi) = \int_t^\infty \scal{\bigl(\hat \Pi_{(t,x)} \tau\bigr)(s,\cdot) \phi(s,\cdot), dW(s)}\;,
\end{equ}
where the integral on the right is the It\^o integral. 

The goal of this section is to show that our renormalised models 
$(\hat \Pi^{(\eps)}, \hat\Gamma^{(\eps)})$
built from $\zeta_\eps$ defined above converge to the same limit.
We prove this by applying a ``diagonal argument'' as in \cite{CLTKPZ}.

\begin{theorem} \label{theo:diagonal}
Let $\hat Z_\eps=(\hat \Pi^{(\eps)},\hat \Gamma^{(\eps)})$ be the renormalised model
built from $\zeta_\eps$ defined in the previous sections (with the choice of renormalisation
constants given by \eqref{e:def-consts}). 
Let $\hat Z=(\hat \Pi, \hat \Gamma)$ be the It\^o
random model.  Then, as $\eps\to 0$, one has
$\hat Z_\eps \to \hat Z$
in distribution in the space $\MM$ of admissible models for $\TT$. 
\end{theorem}

\begin{proof}
 Since the topology on $\MM$ is
generated by the pseudometrics defined in Section~\ref{sec:Adm-model}, it is sufficient to show that 
\begin{equ}[e:wantedConv]
\lim_{\eps \to 0} \E \$\hat Z_\eps, \hat Z\$ = 0\;,
\end{equ}
where the compact domain that the pseudometric depends on is omitted in the notation since it does not matter much in the sequel.
We have
\begin{equ}
\E \$\hat Z_\eps, \hat Z\$
\le \E \$\hat Z_\eps, \hat Z_{\eps,\bar \eps}\$ + \E \$\hat Z_{\eps,\bar \eps}, \hat Z_{0,\bar \eps}\$ + \E \$\hat Z_{0,\bar \eps}, \hat Z\$\;,
\end{equ}
which is valid for any fixed $\bar \eps > 0$ and for any coupling between $\xi$ and $\zeta_\eps$, 
where $\hat Z_{\eps,\bar\eps}$
is the previously defined ``intermediate'' model,
%built from the noise
%$ \zeta_{\eps,\bar\eps} \eqdef \zeta_\eps*\rho_{\bar\eps}$
%with renormalization constants ........($K,\zeta_\eps,\rho_{\bar\eps}$).......,
and the model $\hat Z_{0,\bar\eps}$ is obtained again in the same
way as $\hat Z_{\eps,\bar\eps}$ but built from the noise $ \zeta_{0,\bar\eps} \eqdef \xi*\rho_{\bar\eps}$,
with renormalization constants defined in the same way as those for $\hat Z_{\eps,\bar\eps}$ except that every $\rho_{\bar\eps}$ is replaced by a delta function.
Therefore, 
\begin{equ}[e:threeTerms]
\lim_{\eps \to 0} \E \$\hat Z_\eps, \hat Z\$ \le \lim_{\bar \eps \to 0} \lim_{\eps \to 0} \bigl(\E \$\hat Z_\eps, \hat Z_{\eps,\bar \eps}\$ + \E \$\hat Z_{\eps,\bar \eps}, \hat Z_{0,\bar \eps}\$ + \E \$\hat Z_{0,\bar \eps}, \hat Z\$\bigr)\;.
\end{equ}

For the last term,
$\lim_{\bar \eps \to 0} \E \$\hat Z_{0,\bar \eps}, \hat Z\$^p = 0$
 for every $p$
was 
already shown in \cite{WongZakai}.
Note that since $\xi$ is Gaussian, all the renormalization constants
associated with the model $\hat Z_{0,\bar \eps}$ containing cumulants
of order higher than two vanish;
the other constants slightly differ from the corresponding ones in \cite{WongZakai}
because our constants for $\hat Z_{0,\bar \eps}$ are defined via the kernel $K$,
but this error vanishes as $\bar\eps\to 0$.

Regarding the first term,  by \cite[Thm~10.7]{Regularity}
and the second bound in Theorem~\ref{prop:mombound},
we obtain the bound
$\E \$\hat Z_{\eps}, \hat Z_{\eps,\bar \eps}\$ \lesssim \bar \eps^{\kappa}$
uniformly over $\eps$ sufficiently small, so that this term also vanishes.

It therefore remains to bound the second term. 
This follows from the same argument as in \cite{CLTKPZ}.
Firstly, one has a central limit theorem for the noise $\zeta_\eps$,
namely for every $\alpha<-\frac32$ the field $\zeta_\eps$ converges in law to space-time white noise $\xi$ in $\CC^\alpha(\R\times S^1)$, see \cite[Prop.~6.1]{CLTKPZ}. 
Therefore 
\begin{equ}[e:convC1]
\lim_{\eps \to 0} \E \|\zeta_{0,\bar \eps} - \zeta_{\eps,\bar \eps}\|_{\CC^1;\KK}^p = 0\;.
\end{equ}
for any bounded domain $\KK$.
Also, the map from 
the space of stationary and almost surely periodic noise
equipped with $L^p$ of $\CC^1$ norm
to the space of random admissible models is continuous.
Therefore one has
\begin{equ}
\lim_{\eps \to 0} \E \$\hat Z_{\eps,\bar \eps}, \hat Z_{0,\bar \eps}\$ = 0\;,
\end{equ}
for every fixed (sufficiently small) $\bar\eps>0$, so that the second term in 
\eqref{e:threeTerms}
also vanishes, thus concluding the proof.
\end{proof}

We finally collect all our results and prove the main theorem.

\begin{proof}[of Theorem~\ref{theo:main}]
By \cite[Thm~3.9]{WongZakai} or \cite[Sec.~7]{Regularity} 
one has for every $\eps>0$ a unique maximal solution $U^\eps$ to the fixed point problem \eqref{e:FP} in $\CD^{\gamma,0}$ with respect to the renormalized model $\hat \Pi^{(\eps)}$, and solution $U$ respect to the It\^o model $\hat \Pi$, and by \cite[Thm~4.7]{WongZakai} we have $\lim_{\eps \to 0} \mathbf P(\|U;U^\eps\|_{\gamma,0}>\eps^\theta)=0$ for every $\theta<\kappa/2$.

Using the identity
\begin{equ}
\bigl(\hat \Pi^{(\eps)}_{z} \tau\bigr)(z) = \bigl(\Pi^{(\eps)}_{z} M_\eps \tau\bigr)(z)\;,
\end{equ}
with $M_\eps$ defined above \eqref{e:choiceell},
together with \eqref{e:def-consts}, \eref{e:defL}, we see that this expression is non-zero
for the symbols $\1$, $\<Xi>$, $\<Xi2>$, $\<Xi3>$, $ \<Xi3b>$, $\<Xi4>$, $\<Xi4b>$, $\<Xi4c>$ and $\<Xi4e>$, where one has
\begin{equs}[0]
\bigl(\hat \Pi^{(\eps)}_{z} \1\bigr)(z) = 1\;,\quad
\bigl(\hat \Pi^{(\eps)}_{z} \<Xi>\bigr)(z) = \xi_\eps(z)\;,\quad
\bigl(\hat \Pi^{(\eps)}_{z} \<Xi2>\bigr)(z) = -C^{(1)}_{\zeta}  / \eps\;,\\
\bigl(\hat \Pi^{(\eps)}_{z} \<Xi3>\bigr)(z) = -C^{(2)}_{\zeta}  / \sqrt\eps\;,\quad
\bigl(\hat \Pi^{(\eps)}_{z} \<Xi3b>\bigr)(z) = -C^{(3)}_{\zeta}  / \sqrt\eps \;, \\
\bigl(\hat \Pi^{(\eps)}_{z} \<Xi4>\bigr)(z) = -c^{(2)}_\zeta \;,\quad
\bigl(\hat \Pi^{(\eps)}_{z} \<Xi4b>\bigr)(z) = -c^{(1)}_\zeta \;,\\
\bigl(\hat \Pi^{(\eps)}_{z} \<Xi4c>\bigr)(z) = -c^{(3)}_\zeta \;,\quad
\bigl(\hat \Pi^{(\eps)}_{z} \<Xi4e>\bigr)(z) = -c^{(4)}_\zeta \;.
\end{equs}

Write $\CR_\eps$ as the reconstruction map with respect to the renormalized model $\hat \Pi^{(\eps)}$.
Combining the above identities with \eref{e:RHS} 
%and the expression \eref{e:defR} for thereconstruction operator, 
it follows that one has the identity
\begin{equs}
\CR_{\eps}\bigl(\hat G(U)\sXi\bigr)(z) &= G(u(z))\xi_\eps(z) -  \frac{C^{(1)}_{\zeta}}{\eps}  G'(u(z))G(u(z)) \\
%&\qquad - c^{(1)} G'(u(z))^3 G(u(z)) - c^{(2)} G''(u(z))G'(u(z))G(u(z))^2 \;.
&- \frac{C^{(2)}_{\zeta}}{\eps^{1/2}} G'(u(z))^2 G(u(z)) - \frac{C^{(3)}_{\zeta}}{\eps^{1/2}} G''(u(z)) G^2(u(z))  \\
&- {1\over 6} c^{(1)}_{\zeta} G'''(u(z)) G^3(u(z))
- c^{(2)}_{\zeta} G'(u(z))^3G(u(z))  \\
&- \Big( {1\over 2} c^{(3)}_{\zeta}
+c^{(4)}_{\zeta}\Big) G''(u(z)) G'(u(z)) G^2(u(z))
\end{equs}
By Theorem~\ref{theo:diagonal}, the limiting model is the It\^o model;
it is proved in \cite[Corollary~6.3]{WongZakai} that $u=\CR U$ constructed with respect to the It\^o model is indeed the classical weak solution interpreted in the It\^o sense (\cite{MR876085,DPZ}) which concludes the proof.
\end{proof}

\appendix
\section{Cumulants and Wick products} \label{sec:app}

We define joint cumulant functions $\fC_n (z_1,\ldots,z_n) $ of a space-time centered random field $\zeta$.
%Given a finite set $B$, we write $\CP(B)$ for the collection of all partitions of $B$ (i.e. all sets $\pi \subset \powerset(B)$ (the power set of $B$) such that $\bigcup \pi = B$ and such that any two distinct elements of $\pi$ are disjoint). 
Given any subset $A$ of space-time points we will define a 
joint cumulant function
$\fC(\bar A)$ of $\zeta$. The definition operates recursively in $|A|$, one sets 
\begin{equ} \label{e:mome2cumu}
\fC \big( A \big) 
\eqdef
\E 
\big( 
\prod_{z\in A} \zeta(z)
\big)
-
\sum_{
\substack{ \pi \in \CP(A)\\
\pi \not = \{A\}}
} 
\prod_{B\in\pi} \fC \big(B\big)\;
\end{equ}
where $\CP(A)$ denotes the set of all partitions of $A$. The key cumulant identity comes from moving all the cumulants of \eqref{e:mome2cumu} to the LHS.

We write $\fC_n$ for the $n$-th joint cumulant function of the field $\zeta$
at $n$ space-time points $A=\{z_1,\dots,z_n\}$: 
\begin{equ} \label{e:def-kappa}
	\fC_n (z_1,\ldots,z_n) %= \Cum\bigl(\{\zeta(z_1),\cdots, \zeta(z_n)\}\bigr) \;.
		= \fC \big(A\big) \;.
\end{equ}
Note that $\fC_{1} = 0$ since $\zeta$ is assumed to be centered. We write $\fC_{n}^{(\eps)}$ for the $n$-th cumulant function of $\zeta_{\eps}$

\begin{definition}\label{def: exp decay condition on cumulants} We say that $\zeta$ has exponentially decaying cumulants if there exists $\theta \in (0,1)$ such that for each $n \in \N$ one has the bound
\begin{equ}\label{eq: exp decay condition on cumulants}
\left|
\fC_{n}(z_{1},\dots,z_{n})
\right|
\lesssim 
\theta^{\diam(z_{1},\dots,z_{n})}
\end{equ}
uniform in $z_{1},\dots,z_{n} \in \R^{d}$ where we set $
\diam(z_{1},\dots,z_{n}) 
\eqdef
\sup_{1 \le i,j \le n} |z_{i} - z_{j}|$.
\end{definition} 

We give an example of a random field $\zeta$ satisfying the above property.

\begin{example}
Let $\mu$ be a finite positive measure on $\CC(\R^{d})$ such that for every $\phi \in \supp(\mu)$ and $p > 1$ one has
\begin{equ}
\int \Big( \sup_{z\in\R^{d}} e^{2|z|_{\s}} |\phi(z)| \Big)^p\,\mu(d\phi) < \infty
\end{equ}
Denote by 
$\hat \mu$ a realisation of the Poisson point process  on
$\CC(\R^{d}) \times \R^{d}$ with intensity $\mu \otimes dz$ and set
\begin{equ}
\zeta(z) = \int \phi(z-z')\, \hat \mu(d\phi\otimes dz') - \int \int \phi(z)\,dz \mu(d\phi)\;.
\end{equ}
Then $\zeta$ satisfies Definition \ref{def: exp decay condition on cumulants}.
\end{example} 
We now give the definition of the Wick product of a generically non-Gaussian random field evaluated at a collection of points. 
\begin{definition} \label{def:Wick-product}
Given a collection of space-time points $A$,  %$\{z_i \,:\, i\in A\}$,
the Wick product $\Wick{\prod_{z\in A} \zeta(z)}$ is defined recursively in $|A|$ by setting $\Wick{1} = 1$ and then setting
\begin{equs} \label{e:defWick}
\Wick{\prod_{z\in A} \zeta(z)}
&\eqdef 
\prod_{z\in A} \zeta(z)
-
\sum_{B \subsetneq A}  
\Wick{\prod_{z\in B} \zeta(z)}  
\E\Big( \prod_{z\in A\setminus B} \zeta(z)\Big)   \\
&=
\prod_{z\in A} \zeta(z)
-
\sum_{B \subsetneq A}  \Wick{\prod_{z\in B} \zeta(z)}  
	\sum_{\pi \in \CP(A \setminus B)} 
	  \prod_{\bar B \in \pi}
		\fC(\bar B)  \;.
\end{equs}
\end{definition}

Note that $\Wick{\prod_{z\in B} \zeta(z)}  $ is understood as an operation on $\{\zeta(z)\}_{z \in B}$, not the product. 
The second line of \eqref{e:defWick} follows from the first by applying \eqref{e:mome2cumu}. 
Note that any non-empty Wick product is always mean zero.

The key Wick product identity comes from moving all the subtracted terms of on the RHS of the second line of \eqref{e:defWick} to the first line of the RHS.

We close with a lemma which is sometimes called ``diagram formula" in the literature \footnote{ See \cite{Surgailis1983,GiraitisSurgailis1986} or \cite{CLTKPZ}.} and generalizes \eqref{e:mome2cumu}. It states that moments of Wick products can be calculated by summing over partitions without ``self-contractions".

%As an immediate consequence of Lemma~\ref{lem:WickExp} one has
\begin{lemma} \label{lem:Wick-field}
For every $p\ge 1$ and $m\ge 1$, one has
\begin{equ} [e:Wick-field]
\E \Big( \prod_{k=1}^p 
	\Wick{ \zeta_\eps(x_k^{(1)}) \cdots \zeta_\eps(x_k^{(m)})}  \Big)
	= 	 \sum_{\pi \in \CP_M(M\times P)}  \prod_{B\in\pi} 
	%\kappa^{(\eps)}_{|B|} \Big(\{x^{(i)}_k:(i,k)\in B\} \Big)
	\fC_\eps (B) 
\end{equ}
where $M\eqdef \{i: 1\le i\le m\}$, $P\eqdef \{k: 1\le k\le p\}$, and $\CP_M(M\times P)$ denotes the set of all partitions $\pi \in \CP(M\times P)$ with the property that for every $B \in \pi$ there exist $(i,k),(i',k')\in B$ such that $k\neq k'$ (in particular $|B|>1$).
\end{lemma}

\endappendix

\bibliographystyle{./Martin}
\bibliography{./refs}

\end{document}